\documentclass[11pt]{article}
\usepackage{amsmath,amsthm,amssymb,amsfonts}
\usepackage{latexsym}
\usepackage{graphicx,psfrag,import}
\usepackage{fullpage}
\usepackage{framed}
\usepackage{float}
\usepackage{verbatim} 
\usepackage{color}
\usepackage{epsfig}
\usepackage{epstopdf}
\usepackage{hyperref}
\usepackage{geometry}
\usepackage{mathtools}
\usepackage{multicol}
\usepackage{a4wide}
\usepackage{booktabs}
\usepackage{enumerate}
\usepackage{thmtools}
\usepackage{xr}
\usepackage{lineno}             
\usepackage{epstopdf}
\usepackage{mathrsfs}
\usepackage{caption}
\usepackage{subcaption}
\usepackage{comment}
\usepackage{authblk}
\usepackage{cases}
\usepackage{indentfirst}
\usepackage{amsbsy}
\usepackage{chemfig}
\usepackage{tikz}
\usepackage{color}
\usepackage{graphicx} 
\usepackage{verbatim}
\newtheorem{defn}{Definition}[section]
\newtheorem{rem}[defn]{Remark}
\newtheorem{prop}[defn]{Proposition}

\newtheorem{lem}[defn]{Lemma}

\numberwithin{equation}{section}

\newtheorem{theorem}{Theorem}[section]

\newtheorem{cor}[defn]{Corollary}
\newtheorem{assump}[defn]{Assumption}

\def\G{\mathcal{G}}

         \def\vf{\bv{f}} \def\vF{\bv{F}}  \def\vx{\bv{x}}   \def\vy{\bv{y}}   \def\vX{\bv{X}}   \def\vv{\bv{v}} \def\vw{\bv{w}} 
    \def\vs{\bv{s}}
\def\vn{\bv{n}}
\def\grad{\bv{\nabla}}  

\newcommand{\bv}[1] {\boldsymbol{{#1}}}

\begin{document}
\title{Minimal invariant regions and minimal globally attracting regions for toric differential inclusions}
\author[1]{Yida Ding}
\author[2]{Abhishek Deshpande}
\author[3]{Gheorghe Craciun}
\affil[1]{Department of Mathematics, University of Wisconsin-Madison, {\tt yding54@wisc.edu}.}
\affil[2]{Department of Mathematics, University of Wisconsin-Madison, {\tt deshpande8@wisc.edu}.}
\affil[3]{Department of Mathematics and Department of Biomolecular Chemistry, University of Wisconsin-Madison, {\tt craciun@math.wisc.edu}.}

\maketitle
\begin{abstract}
{\em Toric differential inclusions} occur as key dynamical systems in the context of the {\em Global Attractor Conjecture}. We introduce the notions of minimal invariant regions and minimal globally attracting regions for toric differential inclusions. We describe a procedure for explicitly constructing the minimal invariant and minimal globally attracting regions for two-dimensional toric differential inclusions. In particular, we obtain invariant regions and globally attracting regions for two-dimensional weakly reversible or endotactic dynamical systems (even if they have time-dependent parameters). 
\end{abstract}

\section{Introduction}
A wide range of mathematical models in biology, chemistry, physics, and engineering are governed by interactions between various populations. Often these systems can be represented by a set of differential equations on the positive orthant with polynomial or power-law right-hand sides, i.e., they have the form
\begin{equation}\label{eq:power_law_dyn_system}
\frac{d\vx}{dt}=\sum\limits_{i=1}^m  {\vx}^{\vs_i} \vv_i
\end{equation}
where $\vx = (x_1, x_2, ..., x_n) \in \mathbb{R}^n_{>0}$, $\vs_i, \vv_i\in\mathbb{R}^n$, and $\vx^{\vy      }:=\vx_1^{y_{1}}\vx_2^{y_{2}}...\vx_n^{y_{n}}$. A key dynamical property of such systems is \emph{persistence} which essentially means that no species can go extinct. In particular, a solution $\vx(t)$ of~(\ref{eq:power_law_dyn_system}) is said to be \emph{persistent} if for any initial condition $\vx(0)\in\mathbb{R}^n_{>0}$, we have $\displaystyle\liminf_{t\to\infty}x_i(t)>0$ for $i=1,2,...,n$. The \emph{Persistence Conjecture} says that dynamical systems generated by weakly reversible reaction networks are persistent~\cite{craciun2013persistence}. This conjecture is related to the \emph{Global Attractor Conjecture} which says that there exists a unique globally attracting equilibrium for complex balanced dynamical systems (up to linear conservation laws)~\cite{craciun2009toric}. Many special cases of the Global Attractor Conjecture have been proved~\cite{anderson2011proof,craciun2013persistence,gopalkrishnan2014geometric,pantea2012persistence}. Recently, a proof of the Global Attractor Conjecture in the fully general case was proposed~\cite{craciun2015toric}. A key tool in this proof is the embedding of weakly reversible dynamical systems into \textit{toric differential inclusions}, which are piecewise constant differential inclusions possessing a rich geometric structure~\cite{craciun2015toric,craciun2019polynomial,craciun2019endotactic}. It is known that positive solutions of toric differential inclusions are contained in some specific invariant regions~\cite{craciun2015toric}. In this paper, we focus on two-dimensional toric differential inclusions. We give an explicit procedure for constructing the minimal invariant regions and minimal globally attracting regions for toric differential inclusions. 
In particular, we can interpret our results as follows: two-dimensional complex balanced systems are known to have {\em globally attracting points} located in the strictly positive quadrant; similarly, the corresponding toric differential inclusions have {\em globally attracting compact sets}, which are also subsets of the strictly positive quadrant, and can be characterized in detail. 

This paper is structured as follows: In Section~\ref{sec:E_graph} we introduce reaction networks as Euclidean embedded graphs, and we define the notions of persistence and permanence for dynamical systems generated by reaction networks. In Section~\ref{sec:cones} we rigorously define a polyhedral cone, polyhedral fan and {\em toric differential inclusions}. We also define \textit{minimal invariant regions} and \textit{minimal globally attracting regions} for toric differential inclusions. In Section~\ref{sec:construction}, we present a procedure for constructing the minimal invariant region for a toric differential inclusion $\mathcal{T}_{\mathcal{F},\delta}$, which we denote by $\mathcal{M}_{\mathcal{F},\delta}$. In Section~\ref{sec:min_invariant_region}, we give a proof of correctness for our procedure of constructing the minimal invariant region for a toric differential inclusion. In Section~\ref{sec:min_glob_region}, we show that the region $\mathcal{M}_{\mathcal{F},\delta}$ is also the minimal globally attracting region for a toric differential inclusion. We conclude by summarizing our results and give possible directions for future work.

\section{Euclidean embedded graphs, Persistence, Permanence}\label{sec:E_graph}

A reaction network is a directed graph $\G=(V,E)$ called the Euclidean embedded graph~\cite{craciun2015toric,craciun2019polynomial,craciun2019endotactic}, where $V\subset\mathbb{R}^n$ is the set of vertices and $E$ is the set of edges corresponding to the reactions in the network. We will abbreviate the Euclidean embedded graph by an E-graph. Note that if $\vy,\vy'\in V$ are two distinct vertices of the E-graph $\G$, then $\vy\rightarrow \vy' \in E$ means that there is a directed edge from $\vy$ to $\vy'$. An E-graph $\mathcal{G}=(V,E)$ is called \textit{reversible} if $\vy\rightarrow \vy'\in E$ implies $\vy'\rightarrow \vy\in E$. An E-graph $\mathcal{G}=(V,E)$ is called \textit{weakly reversible} if every edge $\vy\rightarrow \vy'\in E$ is part of some directed cycle. An E-graph $\mathcal{G}=(V,E)$ is called \textit{endotactic} if for every $\vv\in\mathbb{R}^n$ with $\vv\cdot(\vy'-\vy)>0$ for some $\vy\rightarrow \vy'\in E$, there exists a $\tilde{\vy}\rightarrow \tilde{\vy}'\in E$ such that $\vv\cdot(\tilde{\vy}' - \tilde{\vy})<0$ and $\vv\cdot\tilde{\vy} > \vv\cdot \vy$. Given ${\vx}^*\in\mathbb{R}^n_{>0}$, the stoichiometric compatibility class of ${\vx}^*$ is the polyhedron $({\vx}^* + S)\cap\mathbb{R}^n_{>0}$, where $S=\text{span}\{\vy'-\vy \mid \vy\rightarrow \vy'\in E\}$.

Any E-graph $\mathcal{G}=(V,E)$ gives rise to a family of dynamical systems on the positive orthant. Under the standard assumption of  mass-action kinetics~\cite{feinberg1979lectures,gunawardena2003chemical,guldberg1864studies,voit2015150,yu2018mathematical} the dynamical systems generated by $\mathcal{G}$ can be represented as
\begin{eqnarray}\label{eq:autonomous}
\frac{d\vx}{dt}=\sum\limits_{\vy\rightarrow \vy ' \in E} k_{\vy\rightarrow \vy '} {\vx}^{\vy} (\vy ' - \vy) 
\end{eqnarray}
where the parameter $k_{\vy\rightarrow \vy '} > 0$ is the {\em rate constant} corresponding to the reaction $\vy\rightarrow \vy '$. In general, rate constants can actually vary with time due to external signals or forcing, and this gives rise to more general \textit{non-autonomous} dynamical systems of the form
\begin{eqnarray}\label{eq:non_autonomous}
\frac{d\vx}{dt}=\sum\limits_{\vy\rightarrow \vy ' \in E} k_{\vy\rightarrow \vy '}(t) {\vx}^{\vy} (\vy ' - \vy) 
\end{eqnarray}
If there exists an $\epsilon>0$ such that $\epsilon\leq k_{\vy\rightarrow \vy'}(t)\leq\frac{1}{\epsilon}$ for every $\vy\rightarrow \vy'\in E$, then we call such a dynamical system a {\em variable-$k$ polynomial (or power-law) dynamical system}~\cite{craciun2019polynomial,craciun2019endotactic,craciun2013persistence}.

Some of the most relevant properties of these types of systems are expressed by the notions of {\em persistence} and {\em permanence}. 
A dynamical system of the form~(\ref{eq:non_autonomous}) is said to be \emph{persistent} if for any initial condition $\vx_0\in\mathbb{R}^n_{>0}$, the solution $\vx(t)$ satisfies
\begin{eqnarray}
\displaystyle\liminf_{t\rightarrow T}\vx_i(t)>0 
\end{eqnarray}
for every $i=1,2,...,n$, where $T$ is the maximum time for which $\vx(t)$ is well-defined. A dynamical system of the form~(\ref{eq:non_autonomous}) is said to be \emph{permanent} if for every stoichiometric compatibility class $\mathcal{D}$ there exists a compact set $\mathcal{K}\subset \mathcal{D}$ such that, if $\vx(0)\in \mathcal{D}$, then we have $\vx(t)\in\mathcal{K}$ for all  sufficiently large $t$. A dynamical system given by the form~(\ref{eq:autonomous}) is said to be \emph{complex balanced} if there exists $\tilde{\vx}\in\mathbb{R}^n_{>0}$ such that the following is true for every vertex $\vy\in V$:
\begin{eqnarray}
\displaystyle\sum_{\vy\rightarrow \vy' \in E}k_{\vy\rightarrow \vy'}\tilde{\vx}^{\vy} = \displaystyle\sum_{\vy'\rightarrow \vy\in E}k_{\vy'\rightarrow \vy}\tilde{\vx}^{\vy'}.
\end{eqnarray}
These notions are related to the following open conjectures.  
\begin{enumerate}[i.]
\item \textbf{Persistence conjecture:} Dynamical systems generated by weakly reversible reaction networks are persistent.
\item \textbf{Extended Persistence conjecture:} Variable-$k$ dynamical systems generated by endotactic networks are persistent.
\item \textbf{Permanence conjecture:} Dynamical systems generated by weakly reversible reaction networks are permanent.
\item \textbf{Extended Permanence conjecture:} Variable-$k$ dynamical systems generated by endotactic networks are permanent.
\end{enumerate}
A proof of any one of the conjectures above would also imply the \emph{Global Attractor Conjecture}, which says that complex balanced dynamical systems have  a  globally attracting point within any stoichiometric compatibility class. In recent years, there have been several attempts towards the resolution of these conjectures. Craciun, Nazarov and Pantea~\cite{craciun2013persistence} have proved that two-dimensional variable-$k$ endotactic dynamical systems are permanent. Pantea has extended this result to show that weakly reversible variable-$k$ dynamical systems with two-dimensional stoichiometric subspace are permanent~\cite{pantea2012persistence}. Anderson has proved the global attractor conjecture for complex balanced dynamical systems consisting of a {\em single linkage class}~\cite{anderson2011proof}. Gopalkrishnan, Miller and Shiu~\cite{gopalkrishnan2014geometric} have extended this result to show that \emph{strongly endotactic} variable-$k$ dynamical systems are permanent. Craciun has proposed a proof of the global attractor conjecture in full generality, using a method based on toric differential inclusions~\cite{craciun2015toric}. In particular, the proof uses the fact that solutions of toric differential inclusions are confined to certain invariant regions. Therefore, it is important to understand more about invariant regions of toric differential inclusions. In this paper, we give  explicit constructions for minimal invariant regions and minimal globally attracting regions of toric differential inclusions in two dimensions.

\section{Polyhedral Cones, Fans and Toric Differential Inclusions}\label{sec:cones}

Recall that a {\itshape polyhedral cone}~\cite{berman1994nonnegative} $C\subset \mathbb{R}^n$ is a set whose elements can be represented as a nonnegative linear combination of a finite set of vectors as follows 
\begin{equation}
C=\left\{\displaystyle\sum\limits_{i=1}^k a_i\vv_i|a_i\geq 0, \vv_i\in \mathbb{R}^n\right\}
\end{equation}
An \emph{affine cone} is a set of the form $\vx_0 + C$, where $\vx_0\in\mathbb{R}^n$ and $C$ is a cone. In what follows, we might refer to affine cones simply as cones. The meaning will be clear from the context.\\
The polar of a cone $C$ will be denoted by $C^o$ and is defined as follows:
\begin{equation}\displaystyle
C^o=\{\vw\in\mathbb{R}^n \mid  \vw\cdot\vx \leq 0\, \rm{for}\, \vx\in C \}
\end{equation}
Figure~\ref{fig:polar_cone} gives a few examples illustrating the polar of a cone.

\begin{figure}[h!]
\centering
\includegraphics[scale=0.38]{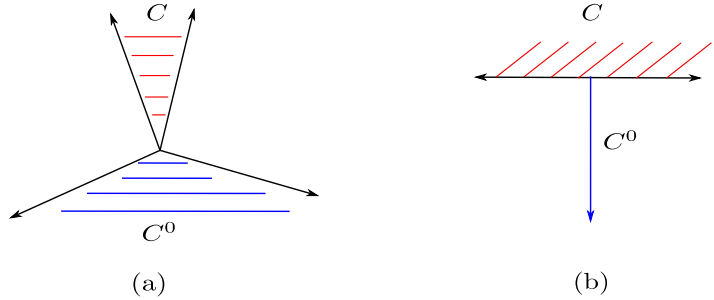}
\caption{\small A few examples showing the polar of a cone $C$. The cone $C$ is marked in red and its polar $C^o$ is marked in blue.}
\label{fig:polar_cone}
\end{figure} 
A \emph{supporting hyperplane} $H$ of a cone $C$ is a hyperplane such that $H\cap C\neq\emptyset$ and $C$ lies in exactly one of the half-spaces generated by $H$. A \emph{face} of a cone is obtained by intersecting the cone with a supporting hyperplane. We now define the notion of a \emph{polyhedral fan}. 

\begin{defn}[Polyhedral Fan]
A finite set $\mathcal{F}$ of polyhedral cones in $\mathbb{R}^n$ is a polyhedral fan if the following two conditions are satisfied:

\begin{enumerate}[(i)]
\item every face of a cone in $\mathcal{F}$ is also a cone in $\mathcal{F}$.
\item the intersection of any two cones in $\mathcal{F}$ is a face of both the cones.
\end{enumerate}
\end{defn}
If $\bigcup\limits_{C\in \mathcal{F}} C=\mathbb{R}^n$, then the polyhedral fan $\mathcal{F}$ is said to be \emph{complete}. Below, we define differential inclusions on the positive orthant. Differential inclusions differ from differential equations in the sense that the right hand side of a differential inclusion is allowed to take values in a set instead of a single point as in the case of differential equations. They are vital to proving the persistence/permanence properties of various dynamical systems. 

\begin{defn}[Differential Inclusion]
A differential inclusion on $\mathbb{R}^n_{>0}$ is a dynamical system of the form
\begin{equation}\displaystyle
			\frac{d\vx}{dt}\in F(\vx)
\end{equation}
, where $F(\vx)\subseteq\mathbb{R}^n$ for all $\vx\in\mathbb{R}^n_{>0}$.
\end{defn}
We now define toric differential inclusions, which are the key dynamical systems of interest in this paper.

\begin{defn}[Toric Differential Inclusions]
Given a complete polyhedral fan $\mathcal{F}$ and $\delta>0$, a toric differential inclusion $\mathcal{T}_{\mathcal{F},\delta}$ is a dynamical system of the form

\begin{eqnarray}
\frac{d\vx}{dt}\in F_{\mathcal{F},\delta}(\vX)
\end{eqnarray}
where $\vX=\log{\vx}\in\mathbb{R}^n$ and 
\begin{eqnarray}\label{eq:initial_tdi}
F_{\mathcal{F},\delta}(\vX) = \left(\displaystyle\bigcap_{\substack{C\in\mathcal{F} \\dist(\vX,C)\leq \delta}} C\right)^o. 
\end{eqnarray}
\end{defn}
From~\cite[Equation 9]{craciun2019quasi}, Equation~(\ref{eq:initial_tdi}) can be written as
\begin{eqnarray}\label{eq:simple_tdi}
F_{\mathcal{F},\delta}(\vX) =  \left({\displaystyle\bigcap_{\substack{C\in\mathcal{F}\\dist(\vX,C)\leq\delta \\ dim(C)=n}} C}\right)^o.
\end{eqnarray}
Figure~\ref{fig:toric_differential_inclusion} depicts the toric differential inclusion for a fan consisting of two line generators. 

\begin{figure}[h!]
\begin{center}
\includegraphics[scale=0.33]{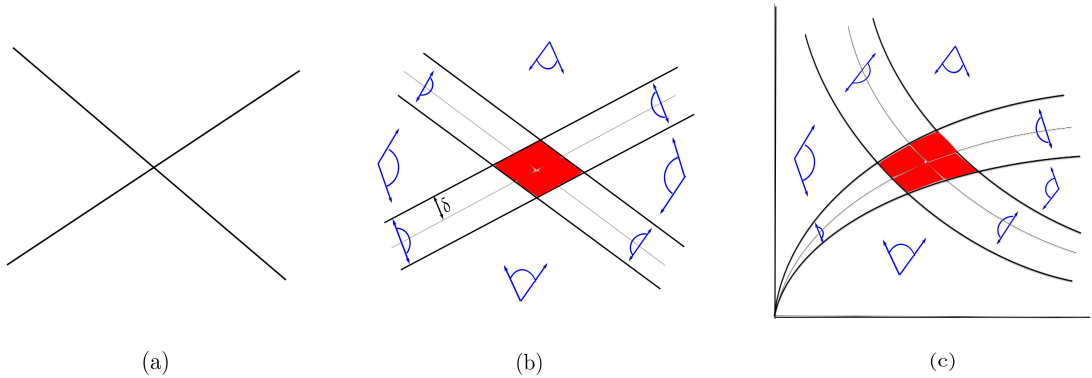}
\caption{\small (a) Fan consisting of nine cones: four cones of dimension two, four cones of dimension one and one cone of dimension zero. (b) Lines parallel to and at a distance $\delta$ from the line generators in (a). The blue cones indicate the right hand side of the toric differential inclusion for regions of space far from the origin. The red part indicates the region of space where the right hand side of the toric differential inclusion is $\mathbb{R}^2$. (c) The curves obtained by exponentiating the lines in (b). The blue cones are unchanged from (b).}
\label{fig:toric_differential_inclusion}
\end{center}
\end{figure}

\begin{defn}[Embedding]
A dynamical system of the form $\frac{d\vx}{dt} = \vf(\vx,t)$ is said to be \textbf{embedded} into the differential inclusion $\frac{d\vx}{dt}\in \vF(\vx)$ if $\vf(\vx,t)\in\vF(\vx)$ for every $\vx\in\mathbb{R}^n_{>0}$ and all $t>0$.
\end{defn}

\begin{defn}[Minimal invariant region]\label{def:invariant_regions}
Consider a toric differential inclusion $\mathcal{T}_{\mathcal{F},\delta}$. A set $\Omega^{\rm{inv}}_{\mathcal{T}_{\mathcal{F},\delta}}\subseteq\mathbb{R}^n_{>0}$ is an {\bf invariant region} of $\mathcal{T}_{\mathcal{F},\delta}$ if for any solution $\vx(t)$ of $\mathcal{T}_{\mathcal{F},\delta}$ with $\vx(0)\in \Omega^{\rm{inv}}_{\mathcal{T}_{\mathcal{F},\delta}}$, we have $\vx(t)\in \Omega^{\rm{inv}}_{\mathcal{T}_{\mathcal{F},\delta}}$ for all $t>0$. A set $\Omega^{\rm{min},\rm{inv}}_{\mathcal{T}_{\mathcal{F},\delta}}$ is the {\bf minimal invariant region} of $\mathcal{T}_{\mathcal{F},\delta}$ if for any invariant region $\Omega^{\rm{inv}}_{\mathcal{T}_{\mathcal{F},\delta}}$, we have $\Omega^{\rm{min},\rm{inv}}_{\mathcal{T}_{\mathcal{F},\delta}}\subseteq \Omega^{\rm{inv}}_{\mathcal{T}_{\mathcal{F},\delta}}$.
\end{defn}

\begin{defn}
Consider a solution $\vx(t)$ of the toric differential inclusion $\mathcal{T}_{\mathcal{F},\delta}$. We call $\vx(t)$ a \textbf{strict solution} of $\mathcal{T}_{\mathcal{F},\delta}$ if for every compact set $K\subset\mathbb{R}^n_{>0}$, there exists $\rho>0$ such that if $\vx(t)\in K$ and $F_{\mathcal{F},\delta}(\log(\vx(t))\neq\mathbb{R}^n$, then $\|\frac{d\vx(t)}{dt}\|>\rho$.
\end{defn}

\begin{defn}[Omega-limit set]\label{def:omega_limit_set}
Consider a toric differential inclusion $\mathcal{T}_{\mathcal{F},\delta}$. Let $\vx(t)$ be a solution of $\mathcal{T}_{\mathcal{F},\delta}$ with $\vx(0)\in \mathbb{R}^n_{>0}$. Then, the \textbf{omega-limit set} of $\vx(t)$ with initial condition $\vx(0)$ is the set $\omega(\vx(0))=\{{\vx}^*\in\mathbb{R}^n_{\geq 0}: \text{there exists a sequence of times} \ t_1 < t_2 <...< t_k$ with $\displaystyle\lim_{k\to\infty}t_k=\infty$ such that $\displaystyle\lim_{k\to\infty}\vx(t_k)=\vx^*\}$. 
\end{defn}

\begin{defn}[Minimal globally attracting region]\label{def:globally_attracting_region}
Consider a toric differential inclusion $\mathcal{T}_{\mathcal{F},\delta}$. A set $\Omega^{\rm{glob}}_{\mathcal{T}_{\mathcal{F},\delta}}\subseteq\mathbb{R}^n_{>0}$ is a {\bf globally attracting region} of $\mathcal{T}_{\mathcal{F},\delta}$ if for any strict solution $\vx(t)$ of $\mathcal{T}_{\mathcal{F},\delta}$ with $\vx(0)\in \mathbb{R}^n_{>0}$, we have $\omega(\vx(0)) \subseteq \Omega^{\rm{glob}}_{\mathcal{T}_{\mathcal{F},\delta}}$. A set $\Omega^{\rm{min},\rm{glob}}_{\mathcal{T}_{\mathcal{F},\delta}}$ is the {\bf minimal globally attracting region} of $\mathcal{T}_{\mathcal{F},\delta}$ if for any globally attracting region $\Omega^{\rm{glob}}_{\mathcal{T}_{\mathcal{F},\delta}}$, we have $\Omega^{\rm{min},\rm{glob}}_{\mathcal{T}_{\mathcal{F},\delta}}\subseteq \Omega^{\rm{glob}}_{\mathcal{T}_{\mathcal{F},\delta}}$.
\end{defn}

\begin{defn}[Trajectory]\label{def:trajectory}
Consider a toric differential inclusion $\mathcal{T}_{\mathcal{F},\delta}$ and points $P_1,P_2\in\mathbb{R}^n_{>0}$. We will say that there is a \textbf{trajectory of $\mathcal{T}_{\mathcal{F},\delta}$} from $P_1$ to $P_2$ if there exists a solution $\vx(t)$ of $\mathcal{T}_{\mathcal{F},\delta}$ with $\vx(0)=P_1$ such that for every $\zeta>0$ there exists $t_0>0$ satisfying $||\vx(t_0) - P_2||<\zeta$.
\end{defn}

\section{Constructing $\mathcal{M}_{\mathcal{F},\delta}$}\label{sec:construction}

In what follows, we describe a procedure for the construction of $\mathcal{M}_{\mathcal{F},\delta}$ in the limit of large $\delta$, where $\mathcal{F}$ is a hyperplane-generated complete polyhedral fan in $\mathbb{R}^2$. From here on, we will refer to a hyperplane-generated complete polyhedral fan in $\mathbb{R}^2$ simply as a fan. Moreover, to simplify the notation, we will denote a point $\vx \in {\mathbb R}^2_{>0}$ by $(x,y)$, and a point $\vX \in {\mathbb R}^2$ by $(X,Y)$.

Let us denote the equations of the line generators of $\mathcal{F}$ by $q_i Y=p_i X$, where $1\leq i\leq b$. We will consider lines that are parallel to these line generators and at a distance $\delta$ from them. Their equations are given by

\begin{eqnarray}\label{eq:one_d_line}
\begin{split}
q_i Y &=p_i X\pm\delta_i,\\ 
\text{where}\, \delta_i &= \delta\sqrt{p^2_i + q^2_i}.
\end{split}
\end{eqnarray}
Under the diffeomorphism $x=e^X,y=e^Y$, these lines get transformed to the curves 
\begin{eqnarray}\label{eq:exponential_curves}
y^{q_i} = h_i x^{p_i},\,\, \text{where}\,\,h_i\in\{e^{-\delta_i},e^{\delta_i}\}.
\end{eqnarray}
We will call the region bounded between the curves $y^{q_i}= e^{-\delta_i}x^{p_i}$ and $y^{q_i}=e^{\delta_i} x^{p_i}$ the \emph{uncertainty region} corresponding to the line $q_i Y=p_i X$. 

\begin{rem}\label{rem:(1,1)_uncert}
Consider a fan $\mathcal{F}$. Then the point $(1,1)$ is contained in the interior of every uncertainty region of $\mathcal{F}$.
\end{rem}

\begin{proof}
Consider a line generator of $\mathcal{F}$ given by $q_i Y=p_i X$. The uncertainty region corresponding to this line generator is the area bounded between the curves $y^{q_i}= e^{-\delta_i}x^{p_i}$ and $y^{q_i}=e^{\delta_i} x^{p_i}$. Note that $1\in(e^{-\delta_i},e^{\delta_i})$. Therefore, the point $(1,1)$ belongs to the interior of the uncertainty region corresponding to the line $q_i Y=p_i X$. Since our choice of the uncertainty region was arbitrary, we get that $(1,1)$ is in the interior of every uncertainty region of $\mathcal{F}$.
\end{proof}

\begin{lem}\label{lem:embed_tdi_another}
Consider the fans $\mathcal{F}$ and $\tilde{\mathcal{F}}$ such that the set of line generators of $\tilde{\mathcal{F}}$ are contained in the set of line generators of $\mathcal{F}$. Then, we have $\mathcal{T}_{\tilde{\mathcal{F}},\delta}\subseteq \mathcal{T}_{\mathcal{F},\delta}$, i.e., $F_{\tilde{\mathcal{F}},\delta}(\vX)\subseteq F_{\mathcal{F},\delta}(\vX)$ for every $\vX\in\mathbb{R}^n$.
\end{lem}

\begin{proof}
Let $\vX\in\mathbb{R}^n$. From Equation~(\ref{eq:simple_tdi}), we have
\begin{eqnarray}
F_{\mathcal{F},\delta}(\vX) =  \left({\displaystyle\bigcap_{\substack{C\in\mathcal{F}\\dist(\vX,C)\leq\delta \\ dim(C)=n}} C}\right)^o,
\end{eqnarray}
and
\begin{eqnarray}
F_{\tilde{\mathcal{F}},\delta}(\vX) =  \left({\displaystyle\bigcap_{\substack{\tilde{C}\in\tilde{\mathcal{F}}\\dist(\vX,\tilde{C})\leq\delta \\ dim(\tilde{C})=n}} \tilde{C}}\right)^o.
\end{eqnarray}
Let $dist(\vX,C)\leq\delta$ for some $C\in\mathcal{F}$ such that $dim(C)=n$. Then, because the set of line generators of $\tilde{\mathcal{F}}$ are contained in the set of line generators of $\mathcal{F}$, there exists a unique cone $\tilde{C}\in\tilde{\mathcal{F}}$ with $dim(\tilde{C})=n$ such that $C\subseteq \tilde{C}$. Since $C\subseteq \tilde{C}$ and $dist(\vX,C)\leq\delta$, we have $dist(\vX,\tilde{C})\leq\delta$. In addition, if $dist(\vX,C'')\leq\delta$ for some $C''\in\tilde{\mathcal{F}}$ with $dim(C'')=n$, then there exists atleast one $C'\in \mathcal{F}$ such that $C'\subseteq C'', dim(C')=n$ and $dist(\vX,C')\leq\delta$. This is true because if $dist(\vX,\hat{C})>\delta$ for every $\hat{C}\in\mathcal{F}$ such that $\hat{C}\subseteq C$ and $dim(\hat{C})=n$, then since 
\begin{eqnarray}
C''=\displaystyle\bigcup_{\substack{\hat{C}\subseteq C''\\dim(\hat{C})=n}} \hat{C}
\end{eqnarray}
we get $dist(\vX,C'')>\delta$, a contradiction. This implies that
\begin{eqnarray}
{\displaystyle\bigcap_{\substack{C\in\mathcal{F}\\dist(\vX,C)\leq\delta \\ dim(C)=n}} C} \subseteq {\displaystyle\bigcap_{\substack{\tilde{C}\in\tilde{\mathcal{F}}\\dist(\vX,\tilde{C})\leq\delta \\ dim(\tilde{C})=n}} \tilde{C}}.
\end{eqnarray}
Since our choice of the point $\vX$ was arbitrary, it follows that  
\begin{eqnarray}
F_{\tilde{\mathcal{F}},\delta}(\vX)\subseteq F_{\mathcal{F},\delta}(\vX)
\end{eqnarray}
for every $\vX\in\mathbb{R}^n$, as desired.
\end{proof}

\begin{figure}[h!]
\centering
\includegraphics[scale=0.6]{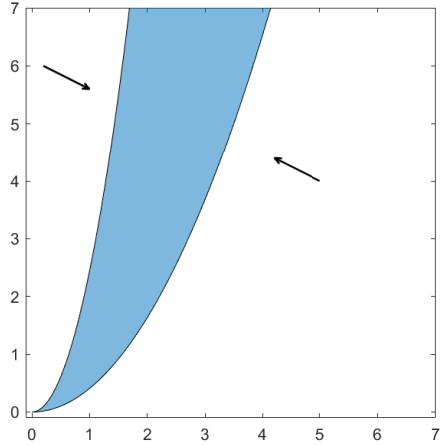}
\caption{\small Attracting directions(shown in black) corresponding to an uncertainty region(shown in blue).}
\label{fig:attracting_direction}
\end{figure}

\begin{defn}[Attracting directions]\label{def:attracting_direction}
Consider a fan $\mathcal{F}$. Let $UC$ be the uncertainty region corresponding to a line generator $qY=pX$ of $\mathcal{F}$. The complement of $UC$ given by $\mathbb{R}^2_{>0}\setminus UC$ consists of two connected components. We define the \textbf{attracting directions} of $UC$ to be the directions that are orthogonal to the line $qY=pX$ and point towards the uncertainty region $UC$ within each connected component. 
\end{defn}
Figure~\ref{fig:attracting_direction} shows the attracting directions corresponding to an uncertainty region.

\begin{defn}\label{def:function_uncertainty_region}
Consider a fan $\mathcal{F}$ and the uncertainty regions corresponding to its one-dimensional generators. Given $\vx\in\mathbb{R}^2_{>0}$, we define $r(\vx)$ to be the number of uncertainty regions that contain the point $\vx$ in their interior.
\end{defn}
In the next remark, we relate the right-hand side of the toric differential inclusion on $\mathbb{R}^2_{>0}$ with the function $r(\vx)$.

\begin{rem}\label{rem:tdi}
The toric differential inclusion $\frac{d\vx}{dt}\in F_{\mathcal{F},\delta}(\vX)$ can be described as follows:
\begin{enumerate}[(i)]		
\item If $r(\vx)=1$, and $\vx$ is in the uncertainty region corresponding to the line $q_i Y=p_iX$. If $ p_i Y + q_i X \geq 0$, then the cone $F_{\mathcal{F},\delta}(\vX)$ is the half space $p_i Y + q_i X\geq 0$. If $p_i Y + q_i X\leq 0$, then the cone $F_{\mathcal{F},\delta}(\vX)$ is the half space $p_i Y + q_i X\leq 0$.
\item If $r(\vx)\geq 2$, then $F_{\mathcal{F},\delta}(\vX)=\mathbb{R}^2$.
\item If $r(\vx)=0$, and $\vx$ lies in the region between the two nearest uncertainty regions which correspond to the lines $q_i Y=p_i X$ and $q_j Y=p_j X$. Then $F_{\mathcal{F},\delta}(\vX)$ is a proper cone formed by the intersection of two half-spaces which are the right-hand sides of the toric differential inclusion corrresponding to these two uncertainty regions.
\end{enumerate}
\end{rem}

\begin{assump}\label{assump:positive_negative_generators}
We will assume that the fan $\mathcal{F}$ has at least one line generator having slope $> 1$, at least one line generator having slope such that $0 < slope <1$ and at least one line generator having slope $< 0$. In addition, we will assume that the line generators of the fan do not have zero or infinite slope. These special cases will be dealt in Section~\ref{sec:special_cases}.
\end{assump}

To build the region $\mathcal{M}_{\mathcal{F},\delta}$, we use the following steps:

\begin{enumerate}

\item[]\textbf{\textit{Step 1: Calculate the intersections of uncertainty regions}}

One can calculate the coordinates corresponding to the intersection of uncertainty regions by solving the following system of equations
\begin{eqnarray}
\begin{split}
y^{q_i}= e^{\pm\delta_i}x^{p_i},\\
y^{q_j}= e^{\pm\delta_j}x^{p_j}
\end{split}
\end{eqnarray}
for every $i\neq j$. Solving them gives us the following points 
\begin{eqnarray}\label{eq:expo_intersection_points}
\begin{split}
(x_1,y_1)&=\bigg(e^{\frac{q_i\delta_j-q_j\delta_i}{p_iq_j-p_jq_i}},e^{\frac{p_i\delta_j-p_j\delta_i}{p_iq_j-p_jq_i}}\bigg), (x_2,y_2)=\bigg(e^{\frac{q_i\delta_j+q_j\delta_i}{p_iq_j-p_jq_i}},e^{\frac{p_i\delta_j+p_j\delta_i}{p_iq_j-p_jq_i}}\bigg),\\ (x_3,y_3)&=\bigg(e^{\frac{-q_i\delta_j-q_j\delta_i}{p_iq_j-p_jq_i}},e^{\frac{-p_i\delta_j-p_j\delta_i}{p_iq_j-p_jq_i}}\bigg), (x_4,y_4)=\bigg(e^{\frac{-q_i\delta_j+q_j\delta_i}{p_iq_j-p_jq_i}},e^{\frac{-p_i\delta_j+p_j\delta_i}{p_iq_j-p_jq_i}}\bigg).
\end{split}
\end{eqnarray}

\begin{rem}\label{rem:para_inter}
Note that substituting $\delta_i=\delta\sqrt{p_i^2 + q_i^2},\delta_j=\delta\sqrt{p_j^2 + q_j^2}$ from Equation~(\ref{eq:one_d_line}) into Equation~(\ref{eq:expo_intersection_points}), we get that the intersection points of the uncertainty regions can be parametrised as $(x,y)=(e^{C\delta},e^{D\delta})$, where $C=f(p_i,p_j,q_i,q_j)$ and $D=g(p_i,p_j,q_i,q_j)$ for some appropriate functions $f$ and $g$.
\end{rem}

Let us denote by $S^{\rm uc}$ the set of all possible intersection points between the uncertainty regions.
Let us assume that the fan $\mathcal{F}$ has line generators given by the set $\mathcal{L} = \{q_iY = p_iX\ |\ i\in[b]\}$, where $[b]=\{1,2,...,b\}$. We can classify the uncertainty regions into the following groups. Let $S_1=\{i\in [b]\ |\ \frac{p_i}{q_i}<0\},S_2=\{i\in[b]\ |\ \frac{p_i}{q_i}>0,\ |q_i|>|p_i|\}, S_3=\{i\in[b]\ |\ \frac{p_i}{q_i}>0,\ |q_i|\leq |p_i|\}$. Note that by assumption~\ref{assump:positive_negative_generators}, each of the sets $S_1,S_2,S_3$ is non-empty. Define 
\begin{eqnarray}
\begin{split}
i_1 &= \arg\displaystyle\max_{i\in S_1}\frac{p_i}{q_i} \\
i_2 &= \arg\displaystyle\min_{i\in S_2}\frac{p_i}{q_i} \\
i_3 &= \arg\displaystyle\max_{i\in S_3}\frac{p_i}{q_i} \\
i_4 &= \arg\displaystyle\min_{i\in S_1}\frac{p_i}{q_i}
\end{split}
\end{eqnarray}
Figure~\ref{fig:construction_step_1} shows the uncertainty regions of the fan $\mathcal{F}$.

\begin{figure}[H]
\begin{center}
\includegraphics[scale=0.6]{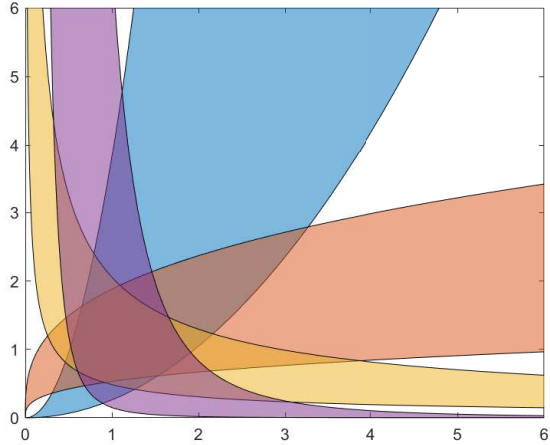}
\caption{Plot depicting the uncertainty regions of the fan $\mathcal{F}$.}
\label{fig:construction_step_1}
\end{center}
\end{figure}

\item[]\textbf{\textit{Step 2: Choose $(N,M)$ and $(n,m)$, the starting points for constructing
$M^{\mathcal{F}}_{\delta}$}}

Let $S^{\rm uc} = \{(x_1,y_1), (x_2,y_2),...,(x_p,y_p))\}$ denote the set of intersection points of the uncertainty regions. Let 
$x_{\max} = \displaystyle\max_{1\leq i\leq p}x_i$ and $y_{\max} = \displaystyle\max_{1\leq i\leq p}y_i$. Let us assume that $y_{\max}=\max{(x_{\max},y_{\max})}$. We will set $M=y_{\max}$ and will denote the intersection point in $S^{\rm uc}$ whose $y$ coordinate is $M$ by $(N,M)$\footnote{If we have several points with same maximal value like $(N_1,M),\ (N_2,M),...$, then choose the intersection point closest to the line $y=x$.}. 

Let $m=\min(\max(x_1,y_1),\max(x_2,y_2),...,\max(x_p,y_p))$. We will denote the intersection point in $S^{\rm uc}$ with this coordinate by $(n,m)$. These two points will serve as the starting points for building the region $\mathcal{M}_{\mathcal{F},\delta}$. Figure~\ref{fig:construction_step_2} illustrates this point.

\begin{figure}[H]
\begin{center}
\includegraphics[scale=0.6]{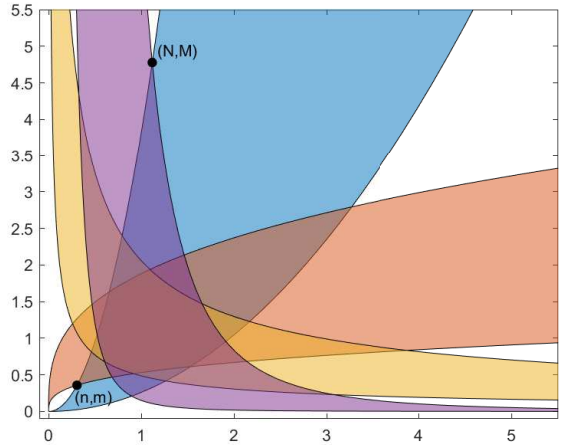}
\caption{Choose starting points $(N,M)$ and $(n,m)$}
\label{fig:construction_step_2}
\end{center}
\end{figure}

\item[]\textbf{\textit{Step 3: Starting from $(N,M)$ build polygonal lines $I_1,I_4$}.}

The procedure described in this step is with respect to Figure~\ref{fig:construction_step_3}. Starting from $A_0=(N,M)$, we build the polygonal line in a counter-clockwise sense as follows. The first line segment that we build is $A_0A_1$ that crosses an uncertainty region such that its slope is given by the slope of the attracting direction of this uncertainty region. Then starting at $A_1$, we repeat this process until one reaches the outer boundary of the uncertainty region with index $i_1$. We will denote these trajectories of polygonal lines by $I_1$.

Then starting from $A_0=(N,M)$ again, we build the polygonal line in a clockwise sense as follows. The first line segment that we build is $A_0B_1$ that crosses an uncertainty region such that its slope is given by the slope of the attracting direction of this uncertainty region. Then starting at $B_1$, we repeat this process until one reaches the outer boundary of the uncertainty region with index $i_4$. We will denote these trajectories of polygonal lines by $I_4$.

\begin{figure}[H]
\begin{center}
\includegraphics[scale=0.6]{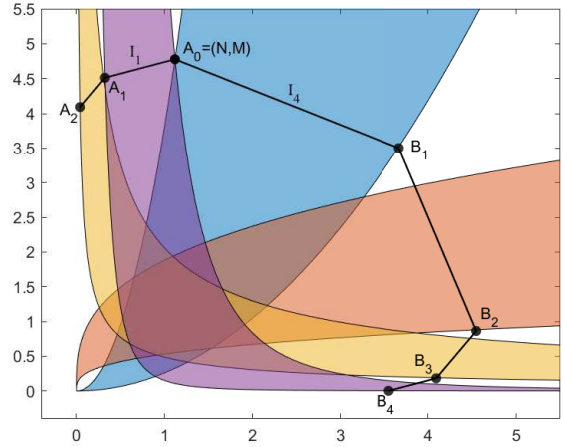}
\caption{Build polygonal lines $I_1,I_4$}
\label{fig:construction_step_3}
\end{center}
\end{figure}

\item[]\textbf{\textit{Step 4: Starting from $(n,m)$ build polygonal lines $I_2,I_3$}}

The procedure described in this step is with respect to Figure~\ref{fig:construction_step_4}. Starting from $C_0=(n,m)$, we build the polygonal line in a clockwise sense as follows. The first line segment that we build is $C_0C_1$ that crosses an uncertainty region such that its slope is given by the slope of the attracting direction of this uncertainty region. Then starting at $C_1$, we repeat this process until one reaches the outer boundary of the uncertainty region with index $i_2$. We will denote these trajectories of polygonal lines by $I_2$.
 
Then starting from $C_0=(n,m)$ again, we build the polygonal line in a counter-clockwise sense as follows. The first line segment that we build is $C_0D_1$ that crosses an uncertainty region such that its slope is given by the slope of the attracting direction of this uncertainty region. Then starting at $D_1$, we repeat this process until one reaches the outer boundary of the uncertainty region with index $i_3$. We will denote these trajectories of polygonal lines by $I_3$.

\begin{figure}[H]
\begin{center}
\includegraphics[scale=0.6]{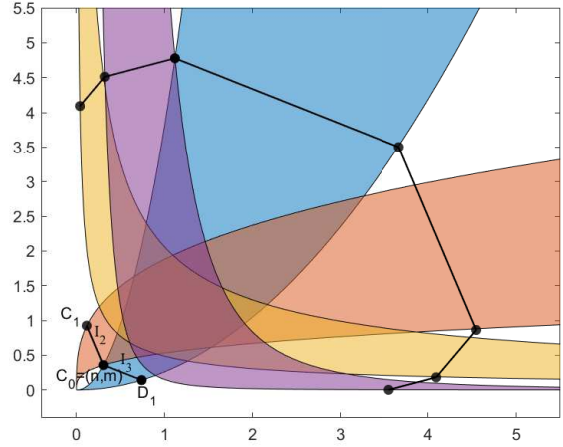}
\caption{Build polygonal lines $I_2,I_3$}
\label{fig:construction_step_4}
\end{center}
\end{figure}

\item[]\textbf{\textit{Step 5: Connect $I_1,I_2,I_3,I_4$}}

We will connect $I_1$ and $I_2$ by following the curves $\mathcal{C}_{i_1}$, $\mathcal{C}_{i_2}$ and will connect $I_3$ and $I_4$ by following the curves  $\mathcal{C}_{i_3},\mathcal{C}_{i_4}$ as shown in Figure~\ref{fig:construction_step_5}. This gives us the boundary of $\mathcal{M}_{\mathcal{F},\delta}$.

\begin{figure}[H]
\begin{center}
\includegraphics[scale=0.6]{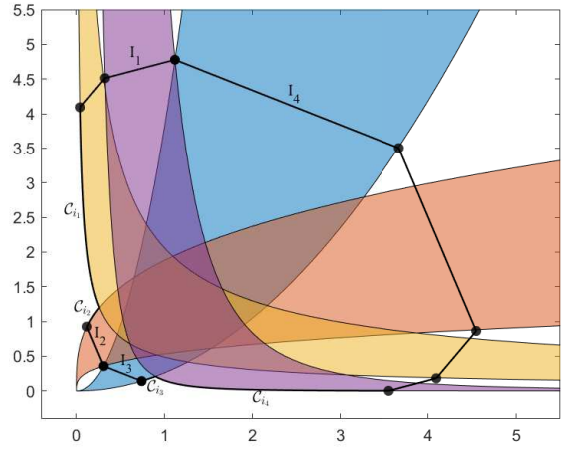}
\caption{Connect $I_1,I_2,I_3,I_4$}
\label{fig:construction_step_5}
\end{center}
\end{figure}

\end{enumerate}

\section{The minimal invariant region}\label{sec:min_invariant_region}

We will assume the following notation with respect to the construction of $M^{\mathcal{F}}_{\delta}$ for the rest of the paper.

\begin{enumerate}

\item Let $C^{\mathcal{F}}_{\delta}$ denote the boundary of $M^{\mathcal{F}}_{\delta}$.

\item The line segments in $C^{\mathcal{F}}_{\delta}$ can be decomposed into the following paths:
\begin{enumerate}[(i)]
\item $\{B_u, B_{u-1}, \cdots , B_p, B_{p-1}, \cdots, B_1, A_0=(N,M), A_1, A_2\cdots, A_q\}$.
\item $\{C_r, C_{r-1}\cdots, C_1, C_0=(n,m), D_1, D_2, \cdots, D_s\}$,
\end{enumerate}
where $A_q$ is the terminal point in the construction of $I_1$ and lies on the uncertainty region with index $i_1$, $B_u$ is the terminal point in the construction of $I_4$ and lies on the uncertainty region with index $i_4$, $C_r$ is the terminal point in the construction of $I_2$ and lies on the uncertainty region with index $i_2$ and $D_s$ is the terminal point in the construction of $I_3$ and lies on the uncertainty region with index $i_3$. Figure~\ref{fig:red_cone} shows these points.

\begin{figure}[h!]
\centering
\includegraphics[scale=0.5]{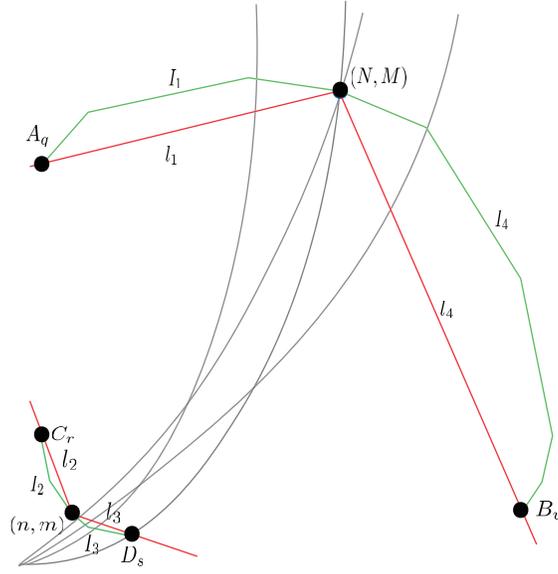}
\caption{\small The green polygonal line denotes the paths $I_1,I_2,I_3,I_4$. The line segments $l_1,l_2,l_3,l_4$ are marked in red. The cone $C_{l_1,l_4}$ is formed by the line segments $l_1$ and $l_4$ with vertex at $(N,M)$. The cone $C_{l_2,l_3}$ is formed by the line segments $l_2$ and $l_3$ with vertex at $(n,m)$.}
\label{fig:red_cone}
\end{figure}

\item We will denote the following
  
\begin{enumerate}
\item[] $l_1$: line segment connecting $(N,M)$ to $A_q$.
\item[] $l_2$: line segment connecting $(n,m)$ to $C_r$.
\item[] $l_3$: line segment connecting $(n,m)$ to $D_s$.
\item[] $l_4$: line segment connecting $(N,M)$ to $B_u$.
\item[] $C_{l_1,l_4}$: cone formed by the line segments $l_1$ and $l_4$ with vertex at $(N,M)$.
\item[] $C_{l_2,l_3}$: cone formed by the line segments $l_2$ and $l_3$ with vertex at $(n,m)$.
\end{enumerate}

Figure~\ref{fig:red_cone} illustrates the cones $C_{l_1,l_4}$ and $C_{l_2,l_3}$ formed by the line segments $l_1,l_2,l_3,l_4$.

\end{enumerate}

\begin{rem}\label{rem:coordinates_points}
Let us denote the coordinates of the point $B_u$ in the construction of $M^{\mathcal{F}}_{\delta}$ as $(x_u,y_u)$. By assumption~\ref{assump:positive_negative_generators}, there is at least one line generator having negative slope and at least one line generator have positive slope. Since $B_u$ is the terminal point in the construction of $I_4$, it lies in the fourth quadrant with respect to Figure~\ref{fig:logarithmic_space}. The point $M$ has the maximum coordinate among the points in $S^{\rm uc}$ and therefore lies either in the first or the second quadrant with respect to Figure~\ref{fig:logarithmic_space}. This implies that $\log(M)>0$ and $\log(y_u)<0$ or equivalently $M>1$ and $y_u<1$.
\end{rem}

\begin{figure}[h!]
\centering
\includegraphics[scale=0.6]{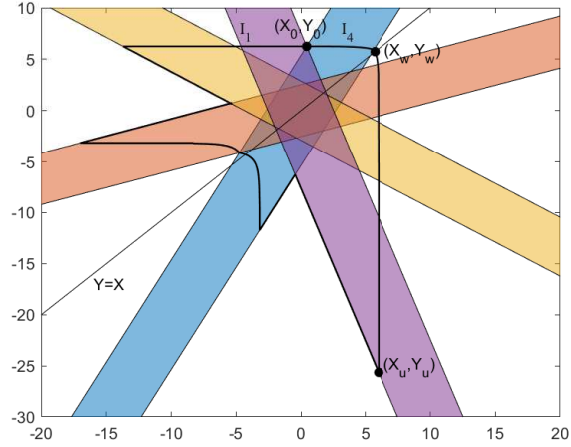}
\caption{\small The curve $C^{\mathcal{F}}_{\delta}$ in logarithmic space.}
\label{fig:logarithmic_space}
\end{figure} 

\begin{rem}\label{rem:knee}
It is instructive to visualize the line segments of $C^{\mathcal{F}}_{\delta}$ in logarithmic space. In particular, consider the diffeomorphism $\phi:(x,y)\to (X,Y)$, where $X=\log(x),Y=\log(y)$. The Jacobian of this diffeomorphism is $J=\begin{pmatrix}
e^{-X} & 0 \\
0 & e^{-Y} 
\end{pmatrix}$. For $Y-X \gg 1$, we consider the rescaled Jacobian $J_1=\begin{pmatrix}
e^{Y-X} & 0 \\
0 & 1 
\end{pmatrix}$. For $X-Y \gg 1$, we consider the rescaled Jacobian $J_2=\begin{pmatrix}
1 & 0 \\
0 & e^{X-Y} 
\end{pmatrix}$
. Note that for $M>N$ and $\delta$ large enough, the rescaled Jacobians $J_1$ and $J_2$ show that the path $I_4$ consists of an (almost) horizontal component and an (almost) vertical component separated by the line $Y=X$ in logarithmic space. Figure~\ref{fig:logarithmic_space} illustrates this fact. Since we assumed $\delta$ to be very large, the shape of $I_4$ in the narrow region around $Y=X$ can be safely ignored.
\end{rem}

\begin{lem}\label{lem:knee_coordinate}
Consider the path $I_4$ that starts from the point $(N,M)$ in the construction of $M^{\mathcal{F}}_{\delta}$. Let us denote the coordinates of the terminal point $B_u$ on $I_4$ by $(x_u,y_u)$. If $N<M$, then in the limit of large $\delta$, we have $x_u=O(M)$.
\end{lem}
\begin{proof}
Let us denote the coordinates of $(N,M)$ by $(x_0,y_0)$ and let $(x_w, y_w)$ be the intersection of $I_4$ with the line $y=x$. Our proof will proceed by analysing the following paths along $I_4$: (i) From $(x_0,y_0)$ to $(x_w, y_w)$ and (ii) From $(x_w,y_w)$ to $(x_u, y_u)$. Let $Q$ denote the set of slopes of the line segments on the path $I_4$ starting from $(x_0,y_0)$ to $(x_w,y_w)$. Now consider the line connecting $(x_0,y_0)$ to $(x_w,y_w)$. Let $Q_{\rm min}$ and $Q_{\rm max}$ denote the minimum and maximum slopes in the set $Q$. Then we have $y_w-y_0=\bar{s}(x_w-x_0)$, where $Q_{\rm min}\leq\bar{s}\leq\ Q_{\rm max}$ and $\bar{s}\neq 1$. By Remark~\ref{rem:para_inter}, we know that the point $(x_0,y_0)$ can be parametrized as $(x_0,y_0)=(e^{N'\delta},e^{M'\delta})$ for some $N',M'\in\mathbb{R}$. Since $x_w=y_w$, we get $y_w-e^{M'\delta}=\bar{s}(y_w-e^{N'\delta})$. This gives
\begin{eqnarray}\label{eq:temp_y}
y_w=\frac{1}{1-\bar{s}}(e^{M'\delta}-\bar{s}e^{N'\delta}).
\end{eqnarray}
Without loss of generality, assume that $y_w=h(x,y,\delta)e^{M'\delta}$ for some function $h(x,y,\delta)$. Inserting this expression of $y_w$ into Equation~(\ref{eq:temp_y}), we get 
\begin{eqnarray}
h(x,y,\delta)=\frac{1}{1-\bar{s}}(1 - \bar{s}e^{(N'-M')\delta})\to \frac{1}{1-\bar{s}}\,\, \text{as}\,\, \delta\to\infty.
\end{eqnarray}
, where the last step follows because $M>N$ or equivalently $M'>N'$. Therefore, $h(x,y,\delta)$ approaches a bounded constant as $\delta\to\infty$. This implies that $x_w=y_w= O(M)$.

Now consider the path along $I_4$ from $(x_w, y_w)$ to $(x_u, y_u)$. Let $Q'$ denote the set of slopes of the line segments on this path. Consider the line connecting $(x_w,y_w)$ to $(x_u,y_u)$. Let $Q'_{\rm min}$ and $Q'_{\rm max}$ denote the minimum and maximum slopes in the set $Q'$. Then we have $y_w-y_u=\tilde{s}(x_w-x_u)$, where $Q'_{\rm min}\leq\tilde{s}\leq\ Q'_{\rm max}$ and $\tilde{s}\neq 1$. Note that since $y_w\neq y_u$, we have $\tilde{s}\neq 0$. Further, since $x_w=y_w$, we have $x_w-y_u=\tilde{s}(x_w-x_u)$. This gives  
\begin{eqnarray}\label{eq:temp_x}
x_w=\frac{1}{1-\tilde{s}}(y_u - \tilde{s}x_u).
\end{eqnarray}
Let us assume that $x_w=\tilde{h}(x,y,\delta)x_u$ for some function $\tilde{h}(x,y,\delta)$. Inserting this expression of $x_w$ into Equation~(\ref{eq:temp_x}), we get
\begin{eqnarray}\label{eq:intermediate_I_4}
\tilde{h}(x,y,\delta)=\frac{1}{1-\tilde{s}}\bigg(\frac{y_u}{x_u} - \tilde{s}\bigg)
\end{eqnarray}
Note from Remark~\ref{rem:knee} that if $M>N$ and for $\delta$ large enough, the path $I_4$ consists of an (almost) horizontal component and an (almost) vertical component separated by the line $Y=X$ in logarithmic space. This implies that $\log(x_u) > \log(y_u)$ or equivalently $x_u>y_u$. Note that $\tilde{h}(x,y,\delta)>0$ for all $x,y\in\mathbb{R}^2_{>0}$ and $\delta > 0$. Therefore, we have
\begin{eqnarray}
\tilde{h}(x,y,\delta) = \bigg|\frac{1}{1-\tilde{s}}\bigg|\bigg|\frac{y_u}{x_u} - \tilde{s}\bigg|\leq \bigg|\frac{1}{1-\tilde{s}}\bigg|\bigg(\bigg|\frac{y_u}{x_u}\bigg| + |\tilde{s}|\bigg)<\bigg|\frac{1}{1-\tilde{s}}\bigg|(1 + |\tilde{s}|).
\end{eqnarray}
This shows that $\tilde{h}(x,y,\delta)$ is upper bounded by a constant. We now show that $\tilde{h}(x,y,\delta)$ is lower bounded by a constant. Consider the following cases:
\begin{enumerate}
\item $\tilde{s} > 1$: Since $\frac{y_u}{x_u} <1$, we get that $\tilde{h}(x,y,\delta)=\frac{1}{1-\tilde{s}}\bigg(\frac{y_u}{x_u} - \tilde{s}\bigg) >1$.
\item $0 < \tilde{s} < \frac{y_u}{x_u} <1$: We show that this case cannot happen. Consider the triangle formed by the points $(0,0),(x_u,y_u),(x_w,y_w)$. The slope of the line segment joining $(0,0)$ to $(x_u,y_u)$ has slope $\frac{y_u}{x_u}$ which is less than 1. In addition, the slope of the line segment joining $(x_u,y_u)$ to $(x_w,y_w)$ has slope $\tilde{s}$ which is less than 1. This implies that the slope of the line segment joining $(0,0)$ to $(x_w,y_w)$ is less than 1, contradicting the fact $x_w=y_w$.
\item $\tilde{s} < 0 < \frac{y_u}{x_u} <1$: Since $\frac{y_u}{x_u}>0$, we get that $\tilde{h}(x,y,\delta)=\frac{1}{1-\tilde{s}}\bigg(\frac{y_u}{x_u} - \tilde{s}\bigg) > \frac{-\tilde{s}}{1-\tilde{s}}$.
\end{enumerate}
Since $\tilde{h}(x,y,\delta)$ is upper bounded and lower bounded by some constants, we get that $x_u = O(x_w) = O(M)$.

\end{proof}

We assume a certain ordering on the uncertainty regions of a fan by looking at the corresponding picture in logarithmic space. The uncertainty region corresponding to the line having the minimum positive slope has index 1. The uncertainty region corresponding to the line having the minimum negative slope has the highest index. The intermediate indices of the uncertainty regions are assigned in increasing order in the counter-clockwise sense. Figure~\ref{fig:ordering_uncert_regions} illlustrates this point.

\begin{figure}[H]
\centering
\begin{subfigure}{0.47\textwidth}
\includegraphics[scale=0.5]{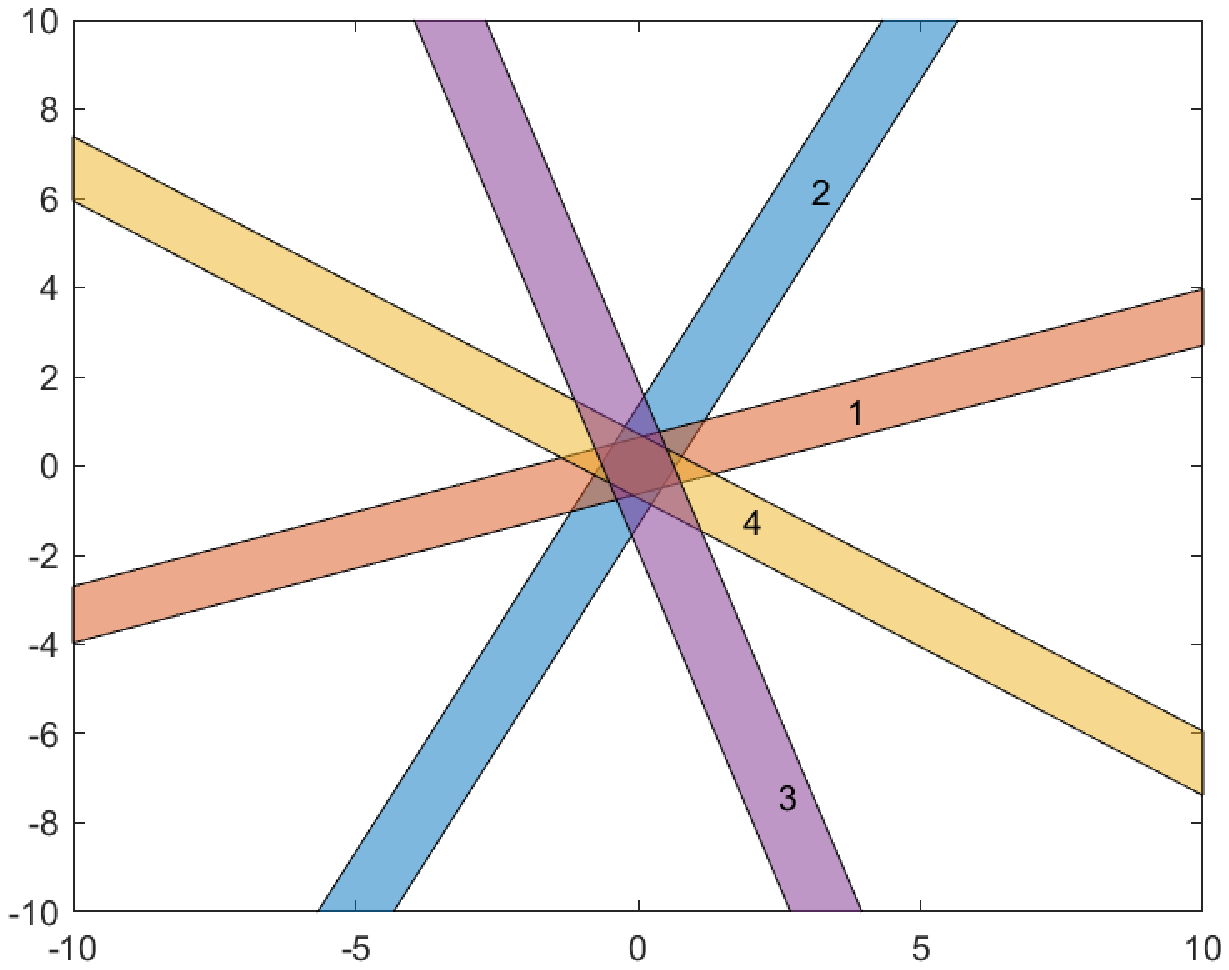}
\caption{\small{Fan with strips of width $\delta$ drawn around its line generators.}}
\label{fig:rotation_log}
\end{subfigure}
\begin{subfigure}{0.47\textwidth}
\includegraphics[scale=0.5]{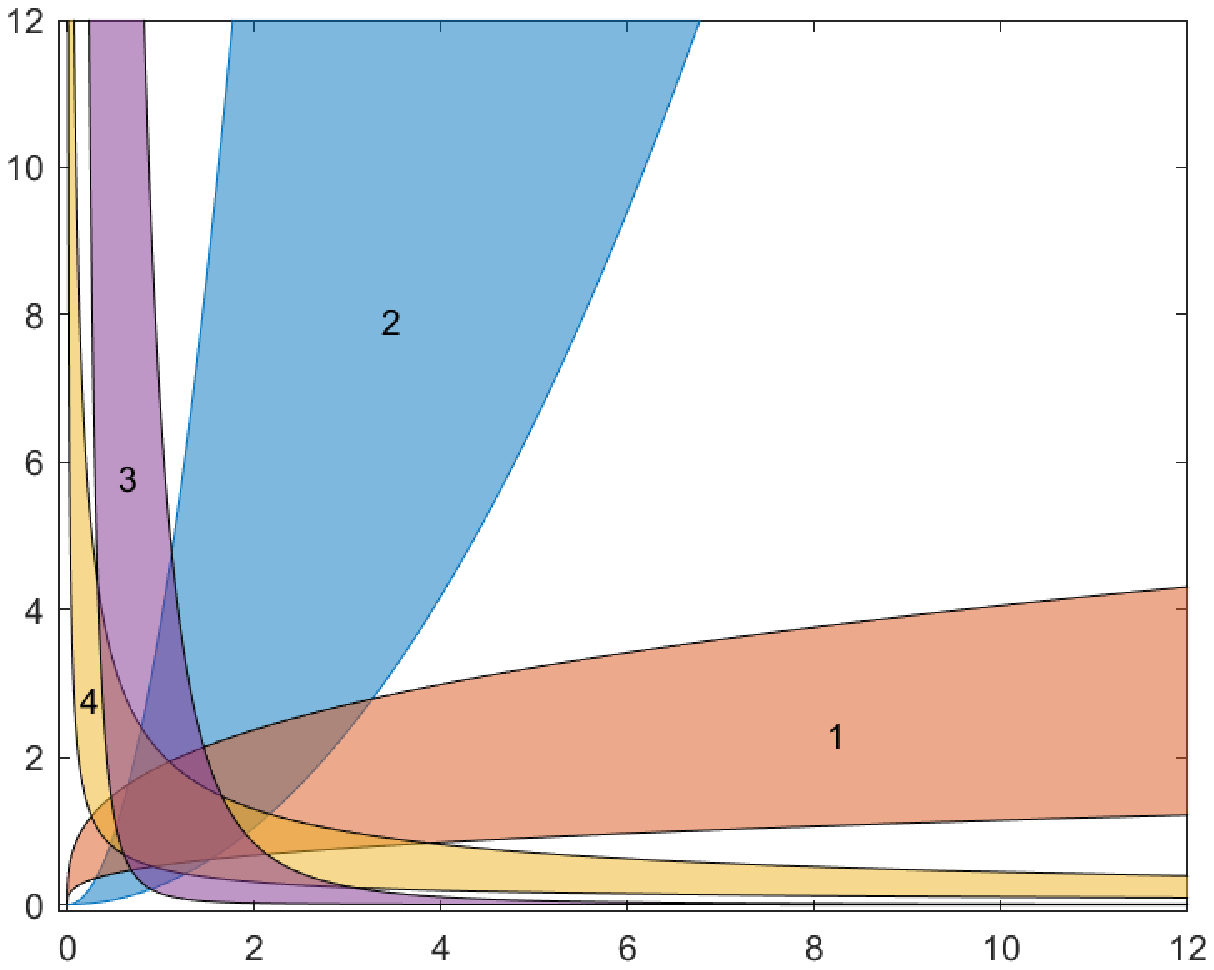}
\caption{\small{This figure is obtained by exponentiating the lines in Figure~\ref{fig:rotation_log}.}} 
\label{fig:rotation_exp}
\end{subfigure}
\caption{\small{We choose a natural ordering on the fat lines in Figure~\ref{fig:rotation_log} starting from the line of the smallest positive slope to the line of the smallest negative slope in a counter-clockwise sense. This induces a corresponding ordering on the uncertainty regions as depicted in Figure~\ref{fig:rotation_exp}}. The notion of going along the attracting direction of the next uncertainty region implies that we traverse the uncertainty regions in a certain order. This in turn induces an ordering on the slopes of the attracting directions along our path, which is the idea behind Remark~\ref{rem:monotonic}.}  
\label{fig:ordering_uncert_regions}
\end{figure}

\begin{rem}\label{rem:monotonic} 
Note that the construction of $M^{\mathcal{F}}_{\delta}$ proceeds by building line segments that go in the attracting direction of the next uncertainty region starting from the points $A_0=(N,M)$ and $C_0=(n,m)$. This puts a natural order on the slopes of the line segments that constitute $M^{\mathcal{F}}_{\delta}$, as remarked in the caption of Figure~\ref{fig:ordering_uncert_regions}. In particular, let $B_p$ be the point in the construction of $I_4$ with the maximum $x$-coordinate. Then, we have the following:
\begin{enumerate}
\item $slope(l_{B_p,B_{p-1}}) < \cdots < slope(l_{B_1,A_0}) < slope(l_{A_0,A_1}) < \cdots < slope(l_{A_{q-1},A_q})$.
\item $slope(l_{C_r,C_{r-1}}) < \cdots < slope(l_{C_1,C_0}) < slope(l_{C_0,D_1}) < \cdots < slope(l_{D_{s-1},D_s}) < 0$.
\item $0<slope(l_{B_u,B_{u-1}}) < \cdots < slope(l_{B_{p+1},B_p})$.
\end{enumerate}

\end{rem}

\begin{lem}\label{lem:curved_part}
The cone corresponding to the right-hand side of the toric differential inclusion $\mathcal{T}_{\mathcal{F},\delta}$ at any point on the curves $\mathcal{C}_{i_1},\mathcal{C}_{i_2},\mathcal{C}_{i_3},\mathcal{C}_{i_4}$ is contained in the interior of the uncertainty regions of these curves.
\end{lem}

\begin{proof}
We will present a proof for the case of $\mathcal{C}_{i_1}$ and $\mathcal{C}_{i_2}$, the case of $\mathcal{C}_{i_3}$ and $\mathcal{C}_{i_4}$ will follow analogously. Let us denote the intersection point of $\mathcal{C}_{i_1}$ and $\mathcal{C}_{i_2}$ by $P_{i_1,i_2}$. We now show the following:
\begin{enumerate}
\item The slopes of the tangents to the curve starting from the point $A_q$ increase monotonically along the curve until the point $P_{i_1,i_2}$. To see this, note that since we stop building line segments in the construction of $M^{\mathcal{F}}_{\delta}$ at the uncertainty region with index $i_1$, Equation~(\ref{eq:exponential_curves}) gives us that $\mathcal{C}_{i_1}$ lies on the curve of the form $y^q = e^{-\sqrt{p^2+q^2}\delta}x^p$, where $p$ and $q$ are such that $p<0$ and $q>0$. Differentiating the equation of the curve $\mathcal{C}_{i_1}$, we get
\begin{eqnarray}\label{eq:slope_i_1}
y' = \frac{p}{q}y^{(1-\frac{q}{p})}e^{\frac{-\sqrt{p^2+q^2}\delta}{p}}  
\end{eqnarray}
Therefore, the slope of the tangent to the curve $\mathcal{C}_{i_1}$ increases with decrease in $y$ and the desired conclusion follows.

\begin{figure}[h!]
\centering
\includegraphics[scale=0.6]{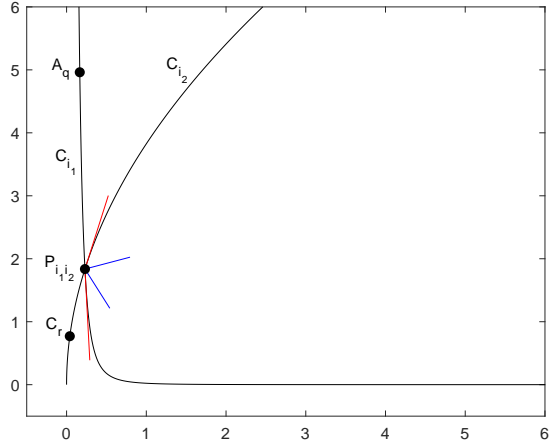}
\caption{\small The right-hand side of the toric differential inclusion $\mathcal{T}_{\mathcal{F},\delta}$ evaluated at the point $P_{i_1,i_2}$ (denoted by the blue cone) is contained in the red cone formed by the tangents to the curves $\mathcal{C}_{i_1}$ and $\mathcal{C}_{i_2}$ at the point $P_{i_1,i_2}$.}
\label{fig:curved_parts}
\end{figure} 

\item The slopes of the tangents to the curve starting from the point $C_r$ decrease monotonically along the curve until the point $P_{i_1,i_2}$. To see this, note that since we stop building line segments in the construction of $M^{\mathcal{F}}_{\delta}$ at the uncertainty region with index $i_2$, Equation~(\ref{eq:exponential_curves}) gives us that $\mathcal{C}_{i_2}$ lies on the curve of the form $y^{\tilde{q}} = e^{\sqrt{{\tilde{p}}^2+{\tilde{q}}^2}\delta}x^{\tilde{p}}$, where $\tilde{p}>0,\tilde{q}>0$ and $\frac{\tilde{q}}{\tilde{p}}>1$. Differentiating the equation of the curve $\mathcal{C}_{i_2}$, we get
\begin{eqnarray}\label{eq:slope_i_2}
y' = \frac{\tilde{p}}{\tilde{q}}y^{(1-\frac{\tilde{q}}{\tilde{p}})}e^{\frac{\sqrt{{\tilde{p}}^2+ {\tilde{q}}^2}\delta}{\tilde{p}}}  
\end{eqnarray} 
Therefore, the slope of the tangent to the curve $\mathcal{C}_{i_2}$ decreases with increase in $y$ and we are done.

Consider Figure~\ref{fig:curved_parts}, which shows a sketch of the regions $\mathcal{C}_{i_1}$ and $\mathcal{C}_{i_2}$. Due to the monotonicity of the slopes of the tangents to $\mathcal{C}_{i_1}$ and $\mathcal{C}_{i_2}$ as argued above, it suffices to show that the cone corresponding to the right-hand side of $\mathcal{T}_{\mathcal{F},\delta}$(marked in blue) lies inside the cone formed by the tangents to the curves $\mathcal{C}_{i_1}$ and $\mathcal{C}_{i_2}$ at the point $P_{i_1,i_2}$(marked in red). Note that from Equations~(\ref{eq:slope_i_1}) and~(\ref{eq:slope_i_2}), we get that $|y'|\to\infty$ as $\delta\to\infty$. Therefore, the slopes of the generators of the red cone can be made as large as possible so that it contains the cone corresponding to the right-hand side of $\mathcal{T}_{\mathcal{F},\delta}$.

\end{enumerate}

\end{proof}

\begin{lem}
The line segments $l_1,l_2,l_3,l_4$ lie inside $M^{\mathcal{F}}_{\delta}$.
\end{lem}

\begin{proof}
We will first show that $l_1$ lies inside $M^{\mathcal{F}}_{\delta}$. Consider the path $\mathcal{P}_1=\{A_q,A_{q-1},\cdots, A_0=(N,M)\}$ as a function $f(x)$ defined on the interval $[x_0,x_1]$ with $A_q = (x_0,f(x_0))$ and $A_0 = (x_1,f(x_1))$. By Remark~\ref{rem:monotonic}, we have the following inequalities: $slope(l_{A_0,A_1}) < \cdots < slope(l_{A_{q-1},A_q})$. This implies that the function $f$ is concave. Therefore, for any two points in the domain of $f$, the function always lies on or above the line segment joining these two points. In particular, if we choose the points to be $A_q$ and $A_0=(N,M)$, we get that  the points $A_q,A_{q-1},...,A_0=(N,M)$ lie on or above the line segment $l_1$, implying that $l_1$ lies inside $M^{\mathcal{F}}_{\delta}$. The same argument can be repeated for the line segments $l_2$ and $l_3$. For the line segment $l_4$, a slight modification of the above argument works by considering the path $\mathcal{P}_4=\{A_0=(N,M),\cdots,B_u\}$ as a function $g(y)$ defined on the interval $[y_0,y_1]$ with $A_0 = (y_0,g(y_0))$ and $B_u = (y_1,g(y_1))$.
\end{proof}

\begin{lem}\label{lem:uc_line_segment_slope}
Let $l^{a,b}_{N,M}$ and $l^{a,b}_{n,m}$ denote the line segments connecting the points $(N,M)$ to $(a,b)$ and $(n,m)$ to $(a,b)$ respectively, where $(a,b)\in S^{\rm{uc}}$. Note from Remark~\ref{rem:para_inter} that $N,M,a,b$ depend on $\delta$. Then
\begin{enumerate}
\item Either $\displaystyle\lim_{\delta\rightarrow\infty}|slope(l^{a,b}_{N,M})|= 0$ for every $(a,b)\in S^{\rm{uc}}$ or $\displaystyle\lim_{\delta\rightarrow\infty}|slope(l^{a,b}_{N,M})|= \infty$ for every $(a,b)\in S^{\rm{uc}}$.
\item Either $\displaystyle\lim_{\delta\rightarrow\infty}|slope(l^{a,b}_{n,m})|= 0$ or $\displaystyle\lim_{\delta\rightarrow\infty}|slope(l^{a,b}_{n,m})|= \infty$ or $\displaystyle\lim_{\delta\rightarrow\infty}|slope(l^{a,b}_{n,m})|= 1$ for $(a,b)\in S^{\rm{uc}}$.
\end{enumerate}
\end{lem}

\begin{proof}
Since $(a,b)\in S^{\rm{uc}},(N,M)\in S^{\rm{uc}}$ and $(n,m)\in S^{\rm{uc}}$, by Remark~\ref{rem:para_inter}, one can parametrize them as follows: $(N,M) = (e^{N'\delta},e^{M'\delta}),(n,m) = (e^{n'\delta},e^{m'\delta})$ and $(a,b)=(e^{a'\delta},e^{b'\delta})$ for some $M',N',m',n',a',b'\in \mathbb{R}$. We split our analysis into the following cases:\\
Case I: Consider the line segment $l^{a,b}_{N,M}$, where $(a,b)\in S^{\rm{uc}}$. Then, we have
\begin{eqnarray}
|slope(l^{a,b}_{N,M})| = \bigg|\frac{e^{M'\delta}-e^{b'\delta}}{e^{N'\delta}-e^{a'\delta}}\bigg|
\end{eqnarray}
If in Step 2 of the construction of $M^{\mathcal{F}}_{\delta}$, $M$ (or equivalently $e^{M'\delta}$) turns out to be the coordinate with the maximum value among points in $S^{\rm uc}$, then we have
\begin{eqnarray}
\displaystyle\lim_{\delta\rightarrow\infty}|slope(l^{a,b}_{N,M})| = \displaystyle\lim_{\delta\rightarrow\infty}\bigg|\frac{e^{M'\delta}-e^{b'\delta}}{e^{N'\delta}-e^{a'\delta}}\bigg| = \displaystyle\lim_{\delta\rightarrow\infty}\bigg|\frac{1-e^{(b'-M')\delta}}{e^{(N'-M')\delta}-e^{(a'-M')\delta}}\bigg| =\infty
\end{eqnarray}
Else, if in Step 2 of the construction of $M^{\mathcal{F}}_{\delta}$, $N$ (or equivalently $e^{N'\delta}$) turns out to be the coordinate with the maximum value among points in $S^{\rm uc}$, then we have 
\begin{eqnarray}
\displaystyle\lim_{\delta\rightarrow\infty}|slope(l^{a,b}_{N,M})| = \displaystyle\lim_{\delta\rightarrow\infty}\bigg|\frac{e^{M'\delta}-e^{b'\delta}}{e^{N'\delta}-e^{a'\delta}}\bigg| = \displaystyle\lim_{\delta\rightarrow\infty}\bigg|\frac{e^{(M'- N')\delta}-e^{(b'- N')\delta}}{1-e^{(a'-N')\delta}}\bigg| = 0.
\end{eqnarray}
Case II: Consider the line segment $l^{a,b}_{n,m}$, where $(a,b)\in S^{\rm{uc}}$. Then, we have
\begin{eqnarray}
|slope(l^{a,b}_{n,m})| = \bigg|\frac{e^{m'\delta}-e^{b'\delta}}{e^{n'\delta}-e^{a'\delta}}\bigg|
\end{eqnarray}
Let us assume that in Step 2 of the construction, we get $m=\min(\max(x_1,y_1),\max(x_2,y_2),...\\,\max(x_p,y_p))$, where $(x_i,y_i)\in S^{\rm uc}$. If $a'<b'$, we have $n'<m'<b'$ and   
\begin{eqnarray}
\displaystyle\lim_{\delta\rightarrow\infty}|slope(l^{a,b}_{n,m})|= \displaystyle\lim_{\delta\rightarrow\infty}\bigg|\frac{e^{m'\delta}-e^{b'\delta}}{e^{n'\delta}-e^{a'\delta}}\bigg| = \displaystyle\lim_{\delta\rightarrow\infty}\bigg|\frac{e^{(m'-b')\delta}-1}{e^{(n'-b')\delta} - e^{(a'-b')\delta}}\bigg| =\infty.
\end{eqnarray}
Else if $a'>b'$, then we have $n' < m' < a'$ and   
\begin{eqnarray}
\displaystyle\lim_{\delta\rightarrow\infty}|slope(l^{a,b}_{n,m})|= \displaystyle\lim_{\delta\rightarrow\infty}\bigg|\frac{e^{m'\delta}-e^{b'\delta}}{e^{n'\delta}-e^{a'\delta}}\bigg| = \displaystyle\lim_{\delta\rightarrow\infty}\bigg|\frac{e^{(m'-a')\delta} - e^{(b'-a')\delta}}{e^{(n'-a')\delta} - 1}\bigg| = 0.
\end{eqnarray}
Else if $a'=b'$, then we have $n' < m' < a'=b'$ and   
\begin{eqnarray}
\displaystyle\lim_{\delta\rightarrow\infty}|slope(l^{a,b}_{n,m})|= \displaystyle\lim_{\delta\rightarrow\infty}\bigg|\frac{e^{m'\delta}-e^{a'\delta}}{e^{n'\delta}-e^{a'\delta}}\bigg| = \displaystyle\lim_{\delta\rightarrow\infty}\bigg|\frac{e^{(m'-a')\delta} - 1}{e^{(n'-a')\delta} - 1}\bigg| = 1.
\end{eqnarray}
If in Step 2 of the construction, we get $n=\min(\max(x_1,y_1),\max(x_2,y_2),...,\max(x_p,y_p))$, a similar calculation shows
that $\displaystyle\lim_{\delta\rightarrow\infty}|slope(l^{a,b}_{n,m})|=\infty$ when $a'>b',\displaystyle\lim_{\delta\rightarrow\infty}|slope(l^{a,b}_{n,m})|= 0$ when $a'<b'$ and $\displaystyle\lim_{\delta\rightarrow\infty}|slope(l^{a,b}_{n,m})|= 1$ when $a'=b'$.

\end{proof}	

The next lemma shows that the cones $C_{l_2,l_3}$ and $C_{l_1,l_4}$ contain certain horizontal and vertical rays. Refer to Figure~\ref{fig:red_cone} for the proof of the next lemma.

\begin{lem}\label{lem:cone_uc}
For $\delta$ large enough, we have the following:
\begin{enumerate}
\item The cone $C_{l_2,l_3}$ contains the set $(n,m)+\mathbb{R}^2_{>0}$.
\item The cone $C_{l_1,l_4}$ either contains a fixed conical neighbourhood of the horizontal ray $(N,M)+\mathbb{R}_{<0}\times\{0\}$ or contains a fixed conical neighbourhood of the vertical ray $(N,M)+\{0\}\times\mathbb{R}_{<0}$ or contains both.
\end{enumerate}
\end{lem}

\begin{proof}
We will first show that the cone $C_{l_2,l_3}$ contains the set $(n,m)+\mathbb{R}^2_{>0}$. Note from Remark~\ref{rem:monotonic} that we have the following set of inequalities: $slope(l_{C_r,C_{r-1}}) < \cdots < slope(l_{C_1,C_0}) < slope(l_{C_0,D_1}) < \cdots < slope(l_{D_{s-1},D_s}) < 0$. Let us denote the point $C_0=(n,m)$ by $(x_0,y_0)$. Let $C_1=(x_1,y_1),C_2=(x_2,y_2),..., C_r=(x_r,y_r)$. Note that $y_0-y_r = \displaystyle\sum_{i=1}^r (y_{i-1} - y_i) = \displaystyle\sum_{i=1}^rslope(l_{C_{i-1,i}})(x_{i-1} - x_i)$. But we also have $y_0-y_r = slope(l_2)(x_0-x_r)= slope(l_2)\displaystyle\sum_{i=1}^r (x_{i-1} - x_i)$. This implies $slope(l_2) = \frac{\displaystyle\sum_{i=1}^rslope(l_{C_{i-1,i}})(x_{i-1} - x_i)}{\displaystyle\sum_{i=1}^r (x_{i-1} - x_i)}$. Note that in the construction of $M^{\mathcal{F}}_{\delta}$, since we build line segments starting from $(n,m)$ towards the point $C_r$ by going along the attracting direction of the next uncertainty region in the clockwise sense, we have $x_i < x_{i-1}$ for $1\leq i\leq r$. Therefore $slope(l_2)$ is a convex combination of the slopes of the line segments $l_{C_1,C_0},l_{C_2,C_1},...,l_{C_{r-1},C_r}$ and is bounded by the maximum slope among these line segments. This implies that $slope(l_2) < slope(l_{C_1,C_0}) < 0$. A similar argument along the path $\{C_0,D_1,...,D_s\}$ gives $slope(l_3) < slope(l_{D_{s-1},D_s}) < 0$. Consequently, the cone $C_{l_2,l_3}$ contains the set $(n,m)+\mathbb{R}^2_{>0}$, as required.

We now show that the cone $C_{l_1,l_4}$ either contains a fixed conical neighbourhood of the horizontal ray $(N,M)+\mathbb{R}_{<0}\times\{0\}$ or contains a fixed conical neighbourhood of the vertical ray $(N,M)+\{0\}\times\mathbb{R}_{<0}$ or contains both. In particular, we show that if in Step 2 of the construction of $M^{\mathcal{F}}_{\delta}$, $M$ turns out to be the coordinate with the maximum value, then $C_{l_1,l_4}$ contains a fixed conical neighbourhood of the vertical ray $(N,M)+\{0\}\times\mathbb{R}_{<0}$; and if $N$ turns out to be the coordinate with the maximum value, then $C_{l_1,l_4}$ contains a fixed conical neighbourhood of the horizontal ray $(N,M)+\mathbb{R}_{<0}\times\{0\}$. We present a proof for the case $M>N$, the other case follows analogously. Let us denote the coordinates of $B_u$ by $(x_u,y_u)$. By Remark~\ref{rem:knee}, we know that in the limit of large $\delta$ and for $M>N$, the path $I_4$ consists of an \emph{almost} horizontal component and an \emph{almost} vertical component separated by the line $Y=X$ in logarithmic space. Figure~\ref{fig:logarithmic_space} depicts the path $I_4$ in logarithmic space. Note that from Remark~\ref{rem:para_inter}, one can parametrize $(N,M)=(e^{N'\delta},e^{M'\delta})$ for some $M',N'\in\mathbb{R}$. By Lemma~\ref{lem:knee_coordinate}, we get that $x_u = c_1M + c_2  =  c_1e^{M'\delta} + c_2$ for some $c_1,c_2>0$. Further, by Remark~\ref{rem:coordinates_points}, we get that $y_u<1$ and $M>1$ or equivalently $M'>0$. This gives 
\begin{eqnarray}\label{eq:negative_slope_generator}
\begin{split}
\displaystyle\lim_{\delta\to\infty}slope(l_4) &= \displaystyle\lim_{\delta\to\infty}\frac{y_u-M}{x_u-N} \\
                                              &= \displaystyle\lim_{\delta\to\infty}\frac{y_u - e^{M'\delta}}{c_1e^{M'\delta} + c_2 - e^{N'\delta}} \\
                                              &= \displaystyle\lim_{\delta\to\infty}\frac{y_u/{e^{M'\delta}} - 1}{c_1 + c_2e^{-M'\delta} - e^{(N'-M')\delta}}\\
                                              &= \frac{-1}{c_1}<0
\end{split}
\end{eqnarray}
Let us denote the coordinates of $A_q$ by $(x_q,y_q)$. Note that in the construction of $M^{\mathcal{F}}_{\delta}$, since we build line segments starting from $(N,M)$ towards the point $A_q$ by going along the attracting direction corresponding to the next uncertainty region in an anticlockwise sense, we get $x_q<N$. This, together with Equation~(\ref{eq:negative_slope_generator}) implies that $C_{l_1,l_4}$ contains a fixed conical neighbourhood of the vertical ray $(N,M)+\{0\}\times\mathbb{R}_{<0}$.
\end{proof}

The next lemma shows that for a sufficiently large $\delta$, the set $\mathcal{M}_{\mathcal{F},\delta}$ contains all the intersection points corresponding to the uncertainty regions of the toric differential inclusion $\mathcal{T}_{\mathcal{F},\delta}$.

\begin{lem}\label{lem:uc_in_invariant}
For large enough $\delta$ we have $S^{\rm uc}\subset M^{\mathcal{F}}_{\delta}$.
\end{lem}

\begin{proof}
Refer to Figure~\ref{fig:red_cone} for this proof. Extend line segments $l_1$ and $l_2$ so that they meet at the point $P_{12}$. Let $R_1$ be the region enclosed between the curves $C_{i_1},C_{i_2}$ and the line segments joining $A_0$ to $P_{12}$ and $C_0$ to $P_{12}$. Similarly, extend line segments $l_3$ and $l_4$ so that they meet at the point $P_{34}$. Let $R_2$ be the region formed between the curves $C_{i_3},C_{i_4}$ and the line segments joining $A_0$ to $P_{34}$ and $C_0$ to $P_{34}$. Note that in the construction of $M^{\mathcal{F}}_{\delta}$, we stop building line segments at the uncertainty regions with indices $i_1,i_2,i_3,i_4$. Therefore, there is no element of $S^{\rm uc}$ in regions $R_1$ and $R_2$. By Lemma~\ref{lem:cone_uc}, we have that for $\delta$ large enough, the cone $C_{l_2,l_3}$ contains the set $(n,m)+\mathbb{R}^2_{>0}$ and the cone $C_{l_1,l_4}$ either contains a fixed conical neighbourhood of the horizontal ray $(N,M)+\mathbb{R}_{<0}\times\{0\}$ or contains a fixed conical neighbourhood of the vertical ray $(N,M)+\{0\}\times\mathbb{R}_{<0}$ or contains both. By Lemma~\ref{lem:uc_line_segment_slope}, for $\delta$ large enough, all the line segments connecting $(N,M)$ to points in $S^{\rm uc}$ have either zero slope or all of them have infinite slope. This implies that all points in $S^{\rm uc}$ are contained in $C_{l_1,l_4}$. In addition, every line segment connecting $(n,m)$ to points in $S^{\rm uc}$ has slope either zero, one or infinite, which implies that all points in $S^{\rm uc}$ are contained in $C_{l_2,l_3}$. Together, these arguments give us that $S^{\rm uc}\in M^{\mathcal{F}}_{\delta}$ for $\delta$ large enough.

\end{proof}

\begin{lem}\label{lem:one_uncertainty}
For large enough $\delta$, we have $r(\vx)\leq 1\ \text{for all}\ \vx\in C_{\delta}^{\mathcal{F}}$.
\end{lem}

\begin{proof} 
For contradiction, assume that there exists a $\vx_0\in C_{\delta}^{\mathcal{F}}$ such that $r(\vx_0)\geq 2$. Since $C_{\delta}^{\mathcal{F}}$ is a continuous curve, there exists a region $D$ that is bounded by at least two different uncertainty regions such that it contains the point $\vx_0$ in its interior. Further, steps $3$ and $4$ in the construction of $\mathcal{M}_{\mathcal{F},\delta}$ ensure that $C_{\delta}^{\mathcal{F}}$ intersects at least one uncertainty region containing $\vx_0$ at both its boundary curves. This line segment of $C_{\delta}^{\mathcal{F}}$ containing $\vx_0$ divides $D$ into two regions. Therefore, there is at least element of $S^{\rm uc}$ that is not contained in $\mathcal{M}_{\mathcal{F},\delta}$, contradicting Lemma~\ref{lem:uc_in_invariant} which says that $S^{\rm uc}\subseteq \mathcal{M}_{\mathcal{F},\delta}$ for a large enough $\delta$.
\end{proof}

\begin{prop}\label{lem:invariant}
For large enough $\delta$, $M^{\mathcal{F}}_{\delta}$ is an invariant region for the toric differential inclusion $\mathcal{T}_{\mathcal{F},\delta}$.
\end{prop}

\begin{proof}
Consider a solution $\vx(t)$ of $\mathcal{T}_{\mathcal{F},\delta}$. We will show that if $\vx(t_0)\in C^{\mathcal{F}}_{\delta}$,  then $\dot{\vx}(t_0)$ points towards the interior of $M^{\mathcal{F}}_{\delta}$~\cite{nagumo1942lage,blanchini1999set}. More precisely, we will show that $\dot{\vx}(t_0)\cdot \vn\leq 0$, where $\vn$ is the outer normal to the curve $C^{\mathcal{F}}_{\delta}$. Note that by Lemma~\ref{lem:one_uncertainty}, we have that $r(\vx)\leq 1$ for all $\vx\in C_{\delta}^{\mathcal{F}}$. If $\vx\in \mathcal{C}_{i_1}\cup \mathcal{C}_{i_2}\cup \mathcal{C}_{i_3}\cup \mathcal{C}_{i_4}$, then Lemma~\ref{lem:curved_part} shows that the cone corresponding to the right-hand side of $\mathcal{T}_{\mathcal{F},\delta}$ points towards the interior of the respective uncertainty regions. This implies that $\dot{\vx}(t_0)\cdot \vn\leq 0$. Else $\vx$ lies on the line segments of $C^{\mathcal{F}}_{\delta}$. We proceed by case analysis:\\
Case 1: $r(\vx(t_0))= 1$. By Remark~\ref{rem:tdi}.(i), the right-hand side of $\mathcal{T}_{\mathcal{F},\delta}$ is a half-space that points towards the interior of $M^{\mathcal{F}}_{\delta}$. Note that the slope of the generators of the half-space is equal to the slope of the line segment of $C^{\mathcal{F}}_{\delta}$ that contains the point $\vx(t_0)$. Therefore, we have $\dot{\vx}(t_0)\cdot \vn\leq 0$.\\
Case 2: $r(\vx(t_0))= 0$. By Remark~\ref{rem:tdi}.(iii), the right-hand side of $\mathcal{T}_{\mathcal{F},\delta}$ is a proper cone formed by the intersection of two half-spaces which are the right-hand sides of the toric differential inclusion corrresponding to  the nearest two uncertainty regions. Note that the half-space generated by the line segment of $C^{\mathcal{F}}_{\delta}$ that points towards $M^{\mathcal{F}}_{\delta}$ is one of these two half-spaces. By the analysis in Case 1, we get that the cone corresponding to the right-hand side of $\mathcal{T}_{\mathcal{F},\delta}$ is contained in both these half-spaces. In particular, it is contained in the half-space generated by the line segment of $C^{\mathcal{F}}_{\delta}$ that points towards $M^{\mathcal{F}}_{\delta}$. This implies that $\dot{\vx}(t_0)\cdot \vn\leq 0$.

\end{proof}

\begin{lem}\label{lem:(1,1)_global_attractor}
For any $\delta>0$ and $\vx_0\in\mathbb{R}^2_{>0}$, there exists a trajectory of $\mathcal{T}_{\mathcal{F},\delta}$ from $\vx_0$ to $(1,1)$.
\end{lem}	

\begin{proof}
Using Proposition~\cite[3.2]{craciun2019polynomial}, there exists a variable-$k$ reversible dynamical system generated by $\mathcal{G}=(V,E)$ with 
\begin{eqnarray}
\epsilon=exp\left({\displaystyle\min_{\vs\rightleftharpoons\vs'\in E}||\vs'-\vs||\frac{\delta}{2}}\right)
\end{eqnarray}
such that it can be embedded into the toric differential inclusion $\mathcal{T}_{\mathcal{F},\delta}$. Setting the rate constants $k_i=1$ for all the reactions in $\mathcal{G}$, we get that $(1,1)$ is a point of complex balanced equilibrium. Since this dynamical system is two-dimensional, the point $(1,1)$ is a global attractor~\cite{craciun2013persistence}. In particular, for any initial condition $\vx(0)\in\mathbb{R}^2_{>0}$, a solution $\vx(t)$ of this reversible dynamical system satisfies $\displaystyle\lim_{t\to\infty}\vx(t)=(1,1)$. This implies that for any $\vx_0\in\mathbb{R}^2_{>0}$, there exists a trajectory of  $\mathcal{T}_{\mathcal{F},\delta}$ from $\vx_0$ to $(1,1)$.
\end{proof}

\begin{cor}\label{cor:(1,1)_min_inv}
The point $(1,1)\in\Omega^{\rm{min,inv}}_{\mathcal{T}_{\mathcal{F},\delta}}$.
\end{cor}

\begin{proof}
This follows from Lemma~\ref{lem:(1,1)_global_attractor}.
\end{proof}

\begin{lem}\label{lem:(M,N)_(m,n)_global_attractor}
For any $\delta>0$ and $\vx_0\in\mathbb{R}^2_{>0}$, there exists a trajectory of $\mathcal{T}_{\mathcal{F},\delta}$ from $\vx_0$ to $(N,M)$.
\end{lem}	

\begin{proof}
Note that from Step 2 in the construction of $M_{\mathcal{F},\delta}$, we have $(N,M)\in S^{\rm uc}$. Without loss of generality, we can assume that $(N,M)$ is  the intersection of curves of the form $y^{q_i}=e^{\delta_i}x^{p_i}$ and $y^{q_j}=e^{\delta_j}x^{p_j}$. Consider the reaction network $\mathcal{G}=(V,E)$, where $E=\{q_i Y\xrightleftharpoons[k_2]{k_1} p_iX, q_j Y\xrightleftharpoons[k_4]{k_3} p_jX\}$. Using Proposition~\cite[3.2]{craciun2019polynomial}, the variable-$k$ reversible dynamical system generated by $\mathcal{G}$ with 
\begin{eqnarray}
\epsilon=exp\left({\displaystyle\min_{\vs\rightleftharpoons\vs'\in E}||\vs'-\vs||\frac{\delta}{2}}\right)
\end{eqnarray}
can be embedded into the toric differential inclusion $\mathcal{T}_{\tilde{\mathcal{F}},\delta}$, where $\tilde{\mathcal{F}}$ is a fan whose line generators are given by $q_i Y = p_i X$ and $q_j Y = p_j X$. Setting the rate constants $k_1=k_3=1,k_2=e^{\delta_i}$ and $k_4=e^{\delta_j}$ for this variable-$k$ reversible dynamical system, we get that $(N,M)$ is a point of complex balanced equilibrium. Since this dynamical system is two-dimensional, the point $(N,M)$ is a global attractor~\cite{craciun2013persistence}. In particular, for any initial condition $\vx(0)\in\mathbb{R}^2_{>0}$, a solution $\vx(t)$ of this reversible dynamical system satisfies $\displaystyle\lim_{t\to\infty}\vx(t)=(N,M)$. Note that since the line generators of $\tilde{\mathcal{F}}$ are included in the line generators of $\mathcal{F}$, by Lemma~\ref{lem:embed_tdi_another} we get that $\mathcal{T}_{\tilde{\mathcal{F}},\delta}\subseteq \mathcal{T}_{\mathcal{F},\delta}$. Therefore, this variable-$k$ reversible dynamical system can be embedded into $\mathcal{T}_{\mathcal{F},\delta}$. This implies that for any $\vx_0\in\mathbb{R}^2_{>0}$, there exists a trajectory from $\vx_0$ to $(N,M)$.
\end{proof}

\begin{lem}\label{lem:uncert_in_invariant}
Consider the toric differential inclusion $\mathcal{T}_{\mathcal{F},\delta}$. Let $UC_1,UC_2$ be any two distinct uncertainty regions of $\mathcal{T}_{\mathcal{F},\delta}$. Then, we have $UC_1\cap UC_2\subseteq\Omega
^{\rm min,\rm{inv}}_{\mathcal{F},\delta}$.
\end{lem}

\begin{proof} 
For contradiction, assume that there exists $\vx_0\in\mathbb{R}^2_{>0}$ such that $\vx_0\in UC_1\cap UC_2$ but $\vx_0\notin\Omega^{\rm min,\rm{inv}}_{\mathcal{F},\delta}$. From Remark~\ref{rem:(1,1)_uncert}, we know that $(1,1)$ is in the interior of $UC_1\cap UC_2$. From Corollary~\ref{cor:(1,1)_min_inv}, we know that $(1,1)\in\Omega
^{\rm min,\rm{inv}}_{\mathcal{F},\delta}$. Note that from Remark~\ref{rem:tdi}.ii, we have $F_{\mathcal{F},\delta}(\vX)=\mathbb{R}^2$ for all points $\vx\in UC_1\cap UC_2$. This implies that for any two points $\vx_1,\vx_2\in UC_1\cap UC_2$, there is a trajectory from $\vx_1$ to $\vx_2$. In particular, there exists a trajectory from $(1,1)$ to $\vx_0$, contradicting the fact that $\vx_0\notin \Omega^{\rm min,\rm{inv}}_{\mathcal{F},\delta}$. 
\end{proof}

\begin{cor}\label{cor:(1,1)_int_min_inv}
The point $(1,1)$ is contained in the interior of $\Omega^{\rm min,\rm{inv}}_{\mathcal{F},\delta}$.
\end{cor}

\begin{proof}
Follows from Remark~\ref{rem:(1,1)_uncert} and Lemma~\ref{lem:uncert_in_invariant}.
\end{proof}

\begin{rem}\label{rem:connected_trajectory}
Consider the points $\vx_0,\vx_1,\vx_2\in\mathbb{R}^2_{>0}$. Note that if there is a trajectory of $\mathcal{T}_{\mathcal{F},\delta}$ from $\vx_0$ to $\vx_1$ and from $\vx_1$ to $\vx_2$, then there exists a trajectory from $\vx_0$ to $\vx_2$. This follows from the fact that solutions of toric differential inclusions depend continuously on their initial conditions.
\end{rem}

\begin{lem} \label{lem:travel_everywhere}
For any two points $\vx_1$ and $\vx_2$ in $\mathcal{M}_{\mathcal{F},\delta}$, there exists a trajectory from $\vx_1$ to $\vx_2$.
\end{lem} 

\begin{proof}
We split our analysis into the following cases:
\begin{enumerate}
\item[] (i) $r(\vx_2)\geq 2$: Note that from Lemma~\ref{lem:(1,1)_global_attractor}, we get that there is a trajectory from $\vx_1$ to $(1,1)$. By Remark~\ref{rem:(1,1)_uncert}, we know that the point $(1,1)$ is contained in every uncertainty region, implying that $r((1,1))\geq 2$. Further, by Remark~\ref{rem:tdi}.(ii), we know that $F_{\mathcal{F},\delta}(\vX)=\mathbb{R}^2$ for every $\vx$ such that $r(\vx)\geq 2$. Therefore, there exists a trajectory from $(1,1)$ to the point $\vx_2$. From Remark~\ref{rem:connected_trajectory}, we get that there is a trajectory from $\vx_1$ to $\vx_2$, as desired.
\item[] (ii) $r(\vx_2)=1$: Let $UC$ denote the uncertainty region that contains the point $\vx_2$. From Lemma~\ref{lem:(M,N)_(m,n)_global_attractor}, we get that there is a trajectory from $\vx_1$ to $(N,M)$. Starting from $(N,M)$, one can go along the boundary $C_{\delta}^{\mathcal{F}}$ till one reaches a point on the line segment intersecting the uncertainty region $UC$. Note that Remark~\ref{rem:tdi}.(i) shows that the right-hand side of the toric differential inclusion at any point inside $UC$ is a half-space that points towards the interior of $\mathcal{M}_{\mathcal{F},\delta}$. Therefore, there exists a trajectory from that point on the line segment intersecting $UC$ to the point $\vx_2$. From Remark~\ref{rem:connected_trajectory}, we get that there that there is a trajectory from $\vx_1$ to $\vx_2$.
\item[] (iii) $r(\vx_2)= 0$: Let $UC_1$ and $UC_2$ denote the two uncertainty regions closest to the point $\vx_2$. From Lemma~\ref{lem:(M,N)_(m,n)_global_attractor}, we get that there is a trajectory from $\vx_1$ to $(N,M)$. Starting from $(N,M)$, one can go along the curve $C_{\delta}^{\mathcal{F}}$ till one reaches a point $P_1$ on the boundary of the uncertainty region $UC_1$ that is closest to the point $\vx_2$. From Remark~\ref{rem:tdi}.(iii), the right-hand side of the toric differential inclusion at $P_1$ is a proper cone formed by the intersection of two half-spaces, which are the right-hand sides of the toric differential inclusion corrresponding to the uncertainty regions $UC_1$ and $UC_2$. Consider the region enclosed between the boundary of the uncertainty regions $UC_1,UC_2$ and $C_{\mathcal{F},\delta}$ such that $r(\vx)=0$ for every point $\vx$ in this region. Note that the right-hand side of the toric differential inclusion at every point in this region is a proper cone formed by the intersection of two half-spaces, which are the right-hand sides of the toric differential inclusion corrresponding to the uncertainty regions $UC_1$ and $UC_2$. In particular, this cone contains the point $\vx_2$. Therefore, there exists a trajectory starting from the point $P_1$ to $\vx_2$. From Remark~\ref{rem:connected_trajectory}, we get that there is a trajectory from $\vx_1$ to $\vx_2$.
\end{enumerate}
\end{proof}
Finally we prove that the region $\mathcal{M}_{\mathcal{F},\delta}$ is the minimal invariant region.
	
\begin{theorem}\label{thm:min_inv_region}
For large enough $\delta$, $\mathcal{M}_{\mathcal{F},\delta}$ is the minimal invariant region for the toric differential inclusion $\mathcal{T}_{\mathcal{F},\delta}$, i.e., $\mathcal{M}_{\mathcal{F},\delta} = \Omega^{\rm{min,inv}}_{\mathcal{T}_{\mathcal{F},\delta}}$.
\end{theorem}

\begin{proof}
By Lemma~\ref{lem:invariant}, we know that $\mathcal{M}_{\mathcal{F},\delta}$ is an invariant region for the toric differential inclusion $\mathcal{T}_{\mathcal{F},\delta}$. We now show that it is minimal, i.e., every invariant region 
contains $\mathcal{M}_{\mathcal{F},\delta}$. Note from the construction of $\mathcal{M}_{\mathcal{F},\delta}$ that the point $(N,M)$ lies on the intersection of two uncertainty regions, i.e., $(N,M)\in S^{\rm uc}$. By Lemma~\ref{lem:uncert_in_invariant}, the point $(N,M)$ must belong to every invariant region. Further, by Lemma~\ref{lem:travel_everywhere}, there exists a trajectory from $(N,M)$ to any point in $\mathcal{M}_{\mathcal{F},\delta}$. This implies that $\mathcal{M}_{\mathcal{F},\delta}$ is contained in every invariant region, as desired.
\end{proof}

\begin{cor}\label{cor:(1,1)_inv}
The point $(1,1)$ is contained in the interior of $\mathcal{M}_{\mathcal{F},\delta}$.
\end{cor}

\begin{proof}
Follows from Corollary~\ref{cor:(1,1)_int_min_inv} and Theorem~\ref{thm:min_inv_region}.
\end{proof}

\section{The minimal globally attracting region}\label{sec:min_glob_region}	

The goal of this section is to show that $\mathcal{M}_{\mathcal{F},\delta}$ is the minimal globally attracting region for the toric differential inclusion $\mathcal{T}_{\mathcal{F},\delta}$. Towards this, we recall the construction of $\mathcal{M}_{\mathcal{F},\delta}$ from Section~\ref{sec:construction}. Note that for $\delta$ large enough, the line segments connecting the polygonal paths $I_1,I_2$ and $I_3,I_4$ approach lines that are parallel to the coordinate axis, making the resultant polygon convex. Let $conv(\mathcal{M}_{\mathcal{F},\delta})$ denote the convex hull of $\mathcal{M}_{\mathcal{F},\delta}$. Therefore, for a sufficiently large $\delta$, $conv(\mathcal{M}_{\mathcal{F},\delta})$ is a closed convex region enclosed by the polygon described above. Figure~\ref{fig:convex} shows $conv(\mathcal{M}_{\mathcal{F},\delta})$ for a sufficiently large $\delta$.

\begin{figure}[H]
\begin{center}
\includegraphics[scale=0.6]{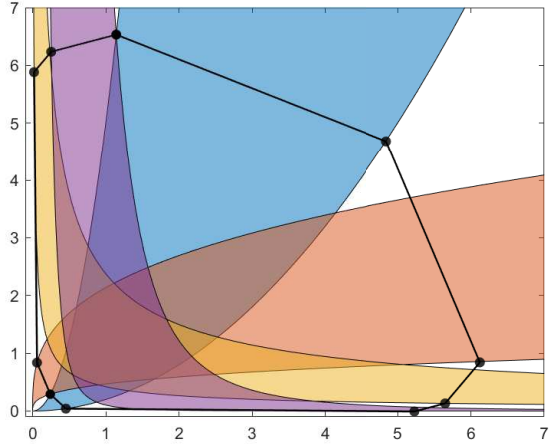}
\caption{The $conv(\mathcal{M}_{\mathcal{F},\delta})$ for large $\delta$.}
\label{fig:convex}
\end{center}
\end{figure}

\begin{theorem}
Given a fan $\mathcal{F}$ and a sufficiently large $\delta$, the region $\mathcal{M}_{\mathcal{F},\delta}$ is the minimal globally attracting region for the toric differential inclusion $\mathcal{T}_{\mathcal{F},\delta}$, i.e. $\Omega^{\rm{min},\rm{glob}}_{\mathcal{T}_{\mathcal{F},\delta}}=\mathcal{M}_{\mathcal{F},\delta}$.
\end{theorem}

\begin{proof}

We first show that $\mathcal{M}_{\mathcal{F},\delta}\subseteq \Omega^{\rm{min},\rm{glob}}_{\mathcal{T}_{\mathcal{F},\delta}}$. Towards this, it is sufficient to show that every point in $\mathcal{M}_{\mathcal{F},\delta}$ is contained in the omega-limit  set of some trajectory of $\mathcal{T}_{\mathcal{F},\delta}$. Note that by Lemma~\ref{lem:(1,1)_global_attractor}, there is a trajectory from $\vx_0\in\mathbb{R}^2_{>0}$ to the point $(1,1)$. In addition, Corollary~\ref{cor:(1,1)_inv} shows that $(1,1)$ is in the interior of $\mathcal{M}_{\mathcal{F},\delta}$. Consider a point $P$ in $\mathcal{M}_{\mathcal{F},\delta}$. By Lemma~\ref{lem:travel_everywhere}, we get that there is a trajectory from $(1,1)$ to $P$. Choose some $\zeta_1>0$. Then there is a trajectory $\vx(t)$ of $\mathcal{T}_{\mathcal{F},\delta}$ with $\vx(0)=(1,1)$, for which there exists $t_1>0$ such that $||\vx(t_1) - P||<\zeta_1$. Now choose some $\zeta'_1>0$. By Lemma~\ref{lem:(1,1)_global_attractor}, there is a trajectory starting from $\vx(t_1)$, for which there exists $t'_1>t_1$ such that $||\vx(t'_1) - (1,1)||<\zeta'_1$. Now choose $0<\zeta_2 < \zeta_1$. Again by Lemma~\ref{lem:travel_everywhere}, we get that there is a trajectory starting from $\vx(t'_1)$ for which there exists $t_2>t'_1>t_1$ such that $||\vx(t_2) - P||<\zeta_2$. Repeating this back and forth between the points $(1,1)$ and $P$, we get a sequence of times $t_1 < t_2 <...< t_k$ with $\displaystyle\lim_{k\to\infty}t_k=\infty$ such that $\displaystyle\lim_{k\to\infty}\vx(t_k)=P$. This implies that the point $P$ belongs to the omega-limit set of the trajectory $\vx(t)$. Since the choice of the point $P$ was arbitrary, the set $\mathcal{M}_{\mathcal{F},\delta}$ is contained in the minimal globally attracting region, i.e. $\mathcal{M}_{\mathcal{F},\delta}\subseteq \Omega^{\rm{min},\rm{glob}}_{\mathcal{T}_{\mathcal{F},\delta}}$, as required.

We now show the other direction, i.e., $\Omega^{\rm{min},\rm{glob}}_{\mathcal{T}_{\mathcal{F},\delta}}\subseteq \mathcal{M}_{\mathcal{F},\delta}$. Towards this, it suffices to show that $\mathcal{M}_{\mathcal{F},\delta}$ is a globally attracting region. Let us denote by $\mathcal{P}$ the boundary of $conv(\mathcal{M}_{\mathcal{F},\delta})$ and by $conv(\mathcal{P})$ the convex hull of $\mathcal{P}$. We will denote by $\mathcal{P}(\delta_0)$ the convex polygon corresponding to $\delta=\delta_0$ that can be constructed using the procedure described in Section~\ref{sec:construction} when $\delta_0$ is large enough. Note that $\mathcal{P}(\delta)$ varies continuously with $\delta$. In addition, we also have that $conv(\mathcal{P}(\delta))$ is a cover of $\mathbb{R}^2_{>0}$, i.e., $\displaystyle\bigcup_{\delta=\delta_0}^{\infty}conv(\mathcal{P}(\delta))=\mathbb{R}^2_{>0}$ and $conv(\mathcal{P}(\delta'))\subset conv(\mathcal{P}(\delta''))$ if $\delta' < \delta''$. Now choose $\gamma$ small enough so that $M_{\delta_0-\gamma}^{\mathcal{F}}$ can still be constructed. Consider a strict solution $\vx(t)$ of the toric differential inclusion $\mathcal{T}_{\delta}^{\mathcal{F}}$ with initial condition $\vx(0)\in\mathbb{R}^2_{>0}$. We will show that $\vx(t)\in conv(\mathcal{P}(\delta_0))$ for a sufficiently large $t$.

Note that since we have $\displaystyle\bigcup_{\delta=\delta_0-\gamma}^{\infty}conv(\mathcal{P}(\delta))=\mathbb{R}^2_{>0}$, one can choose $\delta_1 > \delta_0$ such that $\vx(0)\in\displaystyle\bigcup_{\delta=\delta_0-\gamma}^{\delta_1}conv(\mathcal{P}(\delta))$. Define a function $\Phi: \displaystyle\bigcup_{\delta=\delta_0-\gamma}^{\delta_1}\mathcal{P}(\delta)\to [\delta_0-\gamma,\delta_1]$ such that 
\begin{eqnarray}
\Phi(x,y)=\delta\,\, \text{if}\,\, (x,y)\in\mathcal{P}(\delta).
\end{eqnarray}
We will show that $\Phi(\vx(t))\leq \delta_0$ for $t$ large enough. For contradiction, assume not. Then since $\Phi^{-1}[0,\delta_0]=conv(\mathcal{P}(\delta_0))$ and $\Phi^{-1}[0,\delta_1]=conv(\mathcal{P}(\delta_1))$ are invariant by Lemma~\ref{lem:invariant}, we get $\Phi(\vx(t))\in(\delta_0,\delta_1]$ for all $t$. 

The function $\Phi$ is differentiable on its domain except at the points in the following set $\mathcal{W}=\{B_u, B_{u-1}, B_p, B_{p-1}, \cdots, B_1, A_0=(N,M), A_1, A_2, \cdots, A_q, C_r, C_{r-1}\cdots, C_1, C_0=(n,m), D_1, D_2, \cdots, D_s\}$. Consider $\vx_0\in \mathcal{W}$ . Let $\Phi_1$ and $\Phi_2$ denote the smooth functions that define $\Phi$ on the two line segments in a neighborhood of $\vx_0$. The subgradient of $\Phi$ at $\vx_0$ is~\cite[Definition 8.3]{rockafellar2009variational}
\begin{eqnarray}\label{eq:subgradient}
\displaystyle\partial\Phi(\vx_0)=\{\lambda\grad\Phi_1(\vx_0)+(1-\lambda)\grad\Phi_2(\vx_0)\ |\ \lambda\in[0,1]\}
\end{eqnarray}
This subgradient $\partial \Phi$ exists and is continuous~\cite[Definition 9.1]{rockafellar2009variational}. One can compose $\Phi$ with $\vx(t)$ which is differentiable to get a strictly continuous function $\Phi\circ\vx(t)$. Consequently, one can apply a generalized mean value theorem~\cite[Theorem 10.48]{rockafellar2009variational} to $\Phi\circ\vx(t)$ to get that there exists a $t_0\in[0,t]$ such that
\begin{eqnarray}\label{eq:mean_value_theorem}
\displaystyle\Phi(\vx(t))-\Phi(\vx(0)) = tg_t  \text{ for some } g_t\in \partial (\Phi\circ\vx)(t_0)
\end{eqnarray}
By the chain rule~\cite[Theorem 10.6]{rockafellar2009variational}, we have 
\begin{eqnarray}\label{eq:chain_rule_subgradient}
\partial (\Phi\circ\vx)(t)\subset\{\vw\cdot\dot{\vx}(t)\,|\,\vw\in\partial\Phi(\vx(t))\}.
\end{eqnarray}
Note that Lemma~\ref{lem:invariant} proves that $\mathcal{M}_{\mathcal{F},\delta}$ is invariant by showing that along its boundary, the vector field points towards the interior of $\mathcal{M}_{\mathcal{F},\delta}$. In particular, we can extend this proof to show that for a compact set $J\in\mathbb{R}^2_{>0}$, there exists a $\eta>0$ such that for $\vx(t)\in J$, we have $\grad\Phi_1(\vx(t))\cdot\dot{\vx}(t) < -\eta$ and  $\grad\Phi_2(\vx(t))\cdot\dot{\vx}(t) < -\eta$. Note that from Equation~(\ref{eq:subgradient}), we have that the subgradient $\partial\Phi(\cdot)$ is a convex combination of $\grad\Phi_1(\cdot)$ and $\grad\Phi_2(\cdot)$. This implies that $\displaystyle\partial\Phi(\vx(t))\dot{\vx}(t) < -\eta$. Using Equation~(\ref{eq:chain_rule_subgradient}), we get
\begin{eqnarray}
\sup\limits_{t\geq 0}\partial (\Phi\circ\vx)(t)<-\eta
\end{eqnarray}
From the mean value theorem given by Equation~(\ref{eq:mean_value_theorem}), we get
\begin{eqnarray}
\Phi(\vx(t))<\Phi(\vx(0)) - \eta t
\end{eqnarray}
for all $t>0$, contradicting that $\Phi(\vx(t))\in[\delta_0,\delta_1]$ for all $t\geq 0$. 

Therefore, $\vx(t)\in conv(\mathcal{P}(\delta_0))= conv(M_{\delta_0}^{\mathcal{F}})$ for $t$ large enough. If $\vx(t)\in conv(M_{\delta_0}^{\mathcal{F}})\setminus M_{\delta_0}^{\mathcal{F}} $, then our construction of $M_{\delta_0}^{\mathcal{F}}$ would imply $r(\vx)=0$. Since $\vx(t)$ is a strict solution, we would have $\vx(t_0)\in\mathcal{C}_{i_1}\cup\mathcal{C}_{i_2}\cup\mathcal{C}_{i_3}\cup\mathcal{C}_{i_4}\in M_{\delta_0}^{\mathcal{F}}$ for $t_0$ large enough. This implies that $M_{\delta_0}^{\mathcal{F}}$ is a globally attracting region, as desired.
\end{proof}


\section{Special cases}\label{sec:special_cases}

Note that the construction of $\mathcal{M}_{\mathcal{F},\delta}$ outlined in Section~\ref{sec:construction} makes certain assumptions on the underlying fan; in particular see Assumption~\ref{assump:positive_negative_generators}. In this section, we show how to handle the cases when (i) the line generators of the fan have either all positive slopes or all negative slopes or (ii) at least one line generator of the fan has zero slope or infinite slope. In both these cases, the construction of $\mathcal{M}_{\mathcal{F},\delta}$ proceeds like in Section~\ref{sec:construction}, albeit with certain modifications as we show below.

\begin{enumerate}
\item \emph{If the line generators of the fan have either all positive slopes or all negative slopes}: The procedure for the case with all positive slopes is completely analogous with the procedure outlined in Section~\ref{sec:construction}. We present a procedure for the case with all negative slopes. The starting point $(N,M)$ is chosen as described in Step 2 of Section~\ref{sec:construction}. The starting point $(n,m)$ is chosen differently as shown in Figure~\ref{fig:two_neg}. The construction of $\mathcal{M}_{\mathcal{F},\delta}$ then proceeds like in Section~\ref{sec:construction}.

\begin{figure}[h!]
\begin{center}
\includegraphics[scale=0.6]{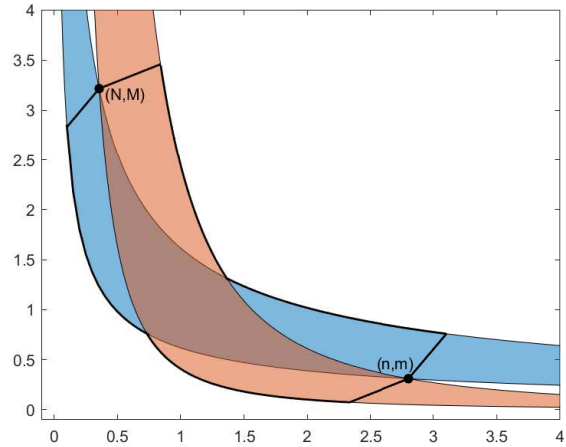}
\caption{The construction of $\mathcal{M}_{\mathcal{F},\delta}$ when the line generators of the fan have all negative slopes.}
\label{fig:two_neg}
\end{center}
\end{figure}

\item \emph{If at least one line generator of the fan has zero or infinite slope}: For example, if one of the line generators of the fan has zero slope, then the construction proceeds exactly as described in Section~\ref{sec:construction} until Step 4. In Step 5, we complete the boundary of $\mathcal{M}_{\mathcal{F},\delta}$ by a vertical line segment joining the polygonal paths $I_1$ and $I_2$ as shown in Figure~\ref{fig:horizontalcase}. Note that it is possible that the construction of $\mathcal{M}_{\mathcal{F},\delta}$ might make either the point $(N,M)$ or $(n,m)$ an interior point of $\mathcal{M}_{\mathcal{F},\delta}$.

\begin{figure}[h!]
\begin{center}
\includegraphics[scale=0.6]{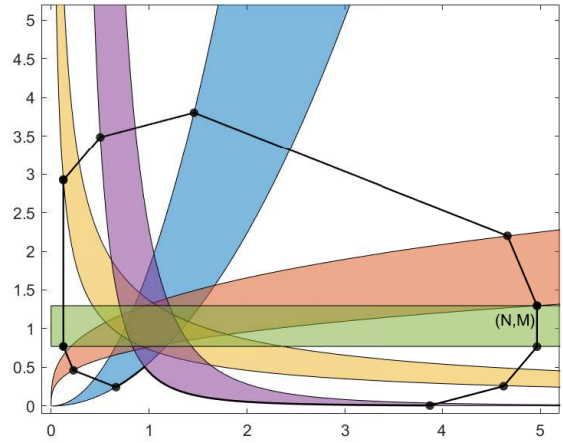}
\caption{The construction of $\mathcal{M}_{\mathcal{F},\delta}$ when a line generator of the fan has zero slope.}
\label{fig:horizontalcase}
\end{center}
\end{figure}

\end{enumerate}

\section{Discussion}\label{sec:discussion}

We have shown how to construct the minimal invariant region for a toric differential inclusion. Additionally, we have shown that the minimal invariant region is also the minimal globally attracting region for the toric differential inclusion. These results are very relevant in the study of mass-action systems and more generally polynomial dynamical systems, since it is known that weakly reversible and endotactic dynamical systems can be embedded into toric differential inclusions~\cite{craciun2019polynomial,craciun2019endotactic}, which is a key step towards the proposed proof of the global attractor conjecture~\cite{craciun2015toric}. It is notable that the structure of toric differential inclusions gives their solutions a greater degree of freedom as compared to the solutions of variable-$k$ mass-action systems. The minimal invariant regions constructed in this paper are also invariant regions for appropriate variable-$k$ mass-action systems. Further, the minimal globally attracting regions allow us to give {\em uniform upper and lower bounds} on the solutions of variable-$k$ mass action systems when $t \to \infty$. 

On the other hand, solutions of variable-$k$ mass-action systems are confined to a {\em proper subset} of the right-hand side of the corresponding toric differential inclusion, since the rate constants of the reactions cannot be switched off completely. It would be interesting to explore how to build minimal invariant and minimal globally attracting regions for variable-$k$ mass-action systems, which we plan to do in upcoming work.  
	 	
\section{Acknowledgements}	
A.D. acknowledges the Van Vleck Visiting Assistant Professorship from the Mathematics Department at University of Wisconsin Madison. The work of G.C. and Y.D. was supported in part by NSF grants DMS-1412643 and DMS-1816238. The work of G.C. was also supported by a Simons fellows grant. 
	 	
\bibliographystyle{amsplain}
\bibliography{Bibliography}

\end{document}